\documentclass[final]{siamltex}
\usepackage{graphicx}
\usepackage{color}
\usepackage{bm}
\usepackage{amsfonts}

\newfont{\bb}{msbm10}

\newcommand{\tr}{^{\sf T}}
\newcommand{\C}[1]{{\cal {#1}}}
\newcommand{\g}[1]{\mbox{\boldmath $#1$}}
\newcommand{\m}[1]{{\bf{#1}}}
\newcommand{\LAMBDA}{{\mbox{\boldmath $\lambda$}}}
\newcommand{\DELTA}{{\mbox{\boldmath $\delta$}}}

\renewcommand{\tilde}{\widetilde}
\renewcommand{\hat}{\widehat}
\renewcommand{\bar}{\overline}

\newtheorem{remark}{Remark}[section]

\begin{document}
\title{An Inexact Accelerated Stochastic ADMM for Separable Convex Optimization
\thanks{January 13, 2020;
revised July 6, 2020.
This research was partially supported by the 
Fundamental Research Funds for the Central Universities of China
under grant G2019KY05106, by the China Postdoctoral Science Foundation,
by the USA National Science Foundation under grants
1522629, 1522654, 1819002, and 1819161, and by the USA Office of Naval Research
under grant N00014-18-1-2100.}
}

\author{
    Jianchao Bai\thanks{{\tt jianchaobai@nwpu.edu.cn},
       School of Mathematics and Statistics
       and the MIIT Key Laboratory of Dynamics and Control of Complex Systems,
       Northwestern Polytechnical University, Xi'an 710072, China.}
\and
    William W. Hager\thanks{{\tt hager@ufl.edu},
        http://people.clas.ufl.edu/hager/,
        PO Box 118105,
        Department of Mathematics,
        University of Florida, Gainesville, FL 32611-8105.
        Phone (352) 294-2308. Fax (352) 392-8357.}
\and
    Hongchao Zhang\thanks{{\tt hozhang@math.lsu.edu},
        https://lsu,
        Department of Mathematics,
        Louisiana State University, Baton Rouge, LA 70803-4918.
        Phone (225) 578-1982. Fax (225) 578-4276.}
}
\maketitle

\begin{abstract}
An inexact accelerated stochastic Alternating Direction Method of Multipliers
(AS-ADMM) scheme is developed for solving structured separable
convex optimization problems with linear constraints.
The objective function is the sum of a possibly nonsmooth convex function
and a smooth function which is an average of many component convex functions.
Problems having this structure often arise in machine learning and 
data mining applications.
AS-ADMM combines the ideas of both ADMM and the stochastic
gradient methods using variance reduction techniques.
One of the ADMM subproblems employs a linearization technique while
a similar linearization could be introduced for the other subproblem.
For a specified choice of the algorithm parameters,
it is shown that the objective error and the constraint violation
are $\mathcal{O}(1/k)$ relative to the number of outer iterations $k$.
Under a strong convexity assumption, the expected iterate error converges
to zero linearly.
A linearized variant of AS-ADMM and incremental sampling strategies are
also discussed.
Numerical experiments with both stochastic and
deterministic ADMM algorithms show that AS-ADMM can be particularly
effective for structured optimization arising in big data applications.
\end{abstract}



\pagestyle{myheadings}
\thispagestyle{plain}
\markboth{J. BAI, W. W. HAGER, and H. ZHANG}
{INEXACT ACCELERATED STOCHASTIC ADMM}

\section{Introduction}
We  consider the following structured separable  convex optimization
problems  with linearly equality constraints:
\begin{equation} \label{P}
\min \{ f(\m{x}) + g(\m{y}) :
\m{x} \in \C{X}, \; \m{y} \in \C{Y}, \; A\m{x} + B\m{y} = \m{b} \} { , }
\end{equation}
where  $\mathcal{X} \subset \mathcal{R}^{n_1}$ and
$ \mathcal{Y} \subset \mathcal{R}^{n_2} $ are  closed  convex  subsets,
$g: \mathcal{Y} \to \mathcal{R} \cup \{+\infty\} $ is a convex,
but not necessarily smooth function, $A\in \mathcal{R}^{n \times n_1}$,
$B\in \mathcal{R}^{n \times n_2}$, and $\m{b} \in \mathcal{R}^{n}$ are  given,
and $f$ is an average of $N$ real-valued convex functions:
\[
f(\m{x})=\frac{1}{N}\sum\limits_{j=1}^{N}f_{j}(\m{x}).
\]
It is assumed that each $f_j$ is defined on an open set containing $\C{X}$
and that $f_j: \mathcal{X} \to \mathcal{R}$ is Lipschitz continuously
differentiable.
Problem (\ref{P}) corresponds to { the} regularized empirical risk minimization
in big data applications, including classification and regression models
in machine learning, where $N$ denotes the  sample size
and $f_{j} $ is the empirical loss.
A major difficulty in problems of the form (\ref{P})
is that $N$ can be very large, and hence, it would be
expensive to evaluate either $f$ or its gradient in each iteration.

The Lagrangian associated with (\ref{P}) is
\begin{equation}\label{Lagr}
\mathcal{L}\left(\m{x}, \m{y},\LAMBDA\right)
=f(\m{x})+g(\m{y})+\LAMBDA\tr
\left(\m{b} -  A\m{x}-B\m{y}\right),
\end{equation}
while the augmented Lagrangian with penalty $\beta > 0$ is
\begin{equation} \label{aug-Lagr}
\mathcal{L}_\beta\left( \m{x}, \m{y},\LAMBDA\right) =
\mathcal{L}\left( \m{x}, \m{y},\LAMBDA\right)+\frac{\beta}{2}
\left\| \m{b} - A\m{x}-B\m{y}\right\|^2.
\end{equation}
The Alternating Direction Method of Multipliers (ADMM) \cite{GM76, GM75}
is an effective approach to exploit the separable structure of the objective
function.
Assuming the existence of a solution to the first-order KKT optimality
system for (\ref{P}), Gabay \cite[pp.~316--322]{GA83}
shows that the following ADMM scheme
\begin{equation}  \label{gs-admm}
\left \{\begin{array}{lll}
\m{x}^{k+1}&\in&\arg\min\limits_{\m{x}\in\mathcal{X}} \mathcal{L}_\beta\left(\m{x},\m{y}^k,\LAMBDA^k\right),\\
\m{y}^{k+1}&\in&\arg\min\limits_{\m{y}\in\mathcal{Y}} \mathcal{L}_\beta\left(\m{x}^{k+1},\m{y},\LAMBDA^k\right),\\
\LAMBDA^{k+1}&=&
\LAMBDA^k+ \beta\left(\m{b} - A\m{x}^{k+1}-B\m{y}^{k+1}\right),
\end{array}\right.
\end{equation}
is a special case of the Douglas-Rachford splitting method
\cite{DouglasRachford1956, Eckstein93b}
applied to the stationary system for the dual of (\ref{P}).
ADMM  was proved convergent for the problem with two-block variables
\cite{GM76}, while the direct extension to more than two blocks
is not necessarily convergent \cite{chyy2016}, although its efficiency has
been observed in some applications \cite{HeYuan18, TaoYuan2011}.

ADMM and its variants have been extensively studied in the literature
and applied to a wide range of applications in signal and image processing,
and in statistical and machine learning. 
Here, we briefly review some of the ADMM literature.
Classes of ADMM-type methods include
proximal ADMM \cite{BaiZhang16,SunSunWang17},
inexact ADMM \cite{GuHeYang14,HagerZhang19, HagerZhang19b, NgWangYuan11}, and
linearized/relaxed ADMM \cite{XuWu11,YangYan2015}.
Most of these are globally convergent with an $\mathcal{O}(1/k)$
ergodic convergence rate, where $k$ denotes the iteration number.
Some improvements in the convergence rate of ADMM have been obtained including
\cite{DLPY2017} where the same $\mathcal{O}(1/k)$ convergence rate is obtained
in a multi-block setting with a Jacobi-proximal implementation.
For either a linear or a quadratic programming problem,
the classic ADMM scheme and its variant have
a linear convergence rate \cite{Boley13}.
Under the assumption that the subdifferential of each component
objective function is  piecewise linear, the global linear convergence of 
ADMM for two-block
separable convex optimization has been established in \cite{YangHan16}.
Assuming that an error bound condition holds and that the dual stepsize
is sufficiently small, Hong and Luo \cite{HongLuo2017} showed
an R-linear convergence rate of their multi-block ADMM.
Under the hypothesis that some of the underlying functions
are strongly convex, global linear convergence of ADMM-type algorithms
and their corresponding proximal/generalized versions
have been established
\cite{CaiHanYuan17,Goldstein14,HanSunZhang17,LinMaZhang15,PengZhang17}.

Notice that in standard deterministic ADMM for (\ref{P}),
gradient methods are often used to solve the subproblem involving $f$.
Hence, the gradient of $f$ needs to be evaluated at each iteration,
which requires the gradient of each component function $f_i$.
This could be expensive or impossible when $N$ is large in
big data applications.
Hence, ADMM type algorithms have been designed in recent years to solve
structured optimization problems of the form (\ref{P}) using
stochastic inexact gradients.
Research in the stochastic gradient ADMM area includes
\cite{AzadiSra14,LiuShangCheng17, Ouyang13, Suzuki13,
XieShanbhag20, ZhaoLiZhou15, ZhongKwok14}.

The algorithm analyzed in this paper is the inexact accelerated ADMM,
denoted AS-ADMM, given in Algorithm~\ref{algo1}.
Note that AS-ADMM contains a routine {\tt xsub} to generate an
approximation to the solution of the $\m{x}$-subproblem in (\ref{gs-admm}),
and two steps corresponds to updates $\m{y}^{k+1}$ and $\g{\lambda}^{k+1}$
in (\ref{gs-admm}).
The algorithm is inexact since the solution of the $\m{x}$-subproblem
is approximated in {\tt xsub}.
The algorithm is stochastic since in each step of {\tt xsub},
the gradient is computed at a randomly chosen component $f_j$ of $f$.
The outcome of AS-ADMM is stochastic since it depends on the randomly
chosen component $f_j$ where the gradient is evaluated.
The structure of AS-ADMM is somewhat typical of the structure for stochastic
gradient ADMM algorithms.

\renewcommand\figurename{Alg.}
\begin{figure}
\begin{center}
{ \tt
\begin{tabular}{l}
\hline
\\
{\bf \textcolor{blue}{Parameters:}}
$\beta>0$, $s \in (0, (1+\sqrt{5})/2 ]$ and $ \C{H} \succ \m{0}. $\\
{\bf \textcolor{blue}{Initialization:}}
$(\m{x}^0,\m{y}^0,\LAMBDA^0)$
$\in\mathcal{X}\times\mathcal{Y}\times \mathcal{R}^n$,
${\breve{\m{x}}^0=\m{x}^0}$.\\
\textbf{For}  $k=0,1,\ldots$ \\
\hspace*{.3in}Choose $M_k, \eta_k >0$ and
$ \C{M}_k $ such that $\C{M}_k - \beta A \tr A  \succeq \m{0}$.\\
\hspace*{.3in}$\m{h}^k :=$
$-A\tr \left[\LAMBDA^k-\beta(A\m{x}^k+B\m{y}^k-\m{b})\right]$. \\
\hspace*{.3in}$(\m{x}^{k+1}, \breve{\m{x}}^{k+1}) =$
{\bf \textcolor{blue}{xsub}} ($\m{x}^k, \breve{\m{x}}^k, \m{h}^k)$. \\
\hspace*{.3in}$\m{y}^{k+1} \in \arg\min \left\{ g(\m{y})+ \frac{\beta}{2}
\left\| A\m{x}^{k+1}+B\m{y}-\m{b}-\frac{\LAMBDA^k}{\beta}
\right\|^2 : \m{y} \in \C{Y} \right\}. $\\
\hspace*{.3in}$\LAMBDA^{k+1}=\LAMBDA^k-
s \beta\left(A\m{x}^{k+1}+B\m{y}^{k+1}-\m{b}\right).$\\
\textbf{end}\\
\hline
\\
$(\m{x}^+, \breve{\m{x}}^+) =$
{\bf \textcolor{blue}{xsub}} ($\m{x}_1$, $\breve{\m{x}}_1$, \m{h}).\\
\textbf{For} $t=1$, 2, $\ldots$, $M_k$ \\
\hspace*{.3in}Randomly select $\xi_t \in \{1, \; 2,\;  \ldots,\;  N\}$
with uniform probability.\\
\hspace*{.3in}$\beta_t = 2/(t+1)$, $\gamma_t = 2/(t\eta_k)$,
$\hat{\m{x}}_t = \beta_t \breve{\m{x}}_t + (1-\beta_t)\m{x}_t$. \\
\hspace*{.3in}$\m{d}_t = \hat{\m{g}}_t + \m{e}_t$, where $\hat{\m{g}}_t = \nabla f_{\xi_t} (\hat{\m{x}}_t)$ and $\m{e}_t$ is a random vector\\
 \hspace*{1.2in}satisfying $\mathbb{E}[\m{e}_t]=\m{0}$. \\
\hspace*{.3in}$\breve{\m{x}}_{t+1}=\arg\min
\left\{\langle \m{d}_t + \m{h},\m{x}\rangle+\frac{\gamma_t}{2}
\left\|\m{x}-\breve{\m{x}}_{t}\right\|_{ \C{H} }^2
+\frac{1}{2} \left\|\m{x}-\m{x}^k \right\|_{ \C{M}_k}^2 :
\m{x} \in \C{X} \right\}.$ \\
\hspace*{.3in}$\m{x}_{t+1}=$
$\beta_t \breve{\m{x}}_{t+1}+(1-\beta_t)\m{x}_{t}$.\\
\textbf{end}\\
\textbf{Return} $(\m{x}^+, \breve{\m{x}}^+) =
(\m{x}_{M_k+1}, \breve{\m{x}}_{M_k+1})$.\\
\hline
\end{tabular}
}
\end{center}
\caption{Accelerated Stochastic ADMM (AS-ADMM)}\label{algo1}
\end{figure}
\renewcommand\figurename{Fig.}

It seems that the first development of a stochastic gradient ADMM scheme
is given in \cite{Ouyang13}.
In the context of (P), the algorithm computes the gradient of a single
randomly chosen component $f_j$, and uses this gradient to linearize $f_j$
at the current iterate.
The solution of the linearized problem yields $\m{x}^{k+1}$.
If $\mathbb{E}$ denotes expectation, $(\m{x}^*,\m{y}^*)$ denotes
a solution of (\ref{P}), and $\C{X}$ is compact, then it is shown that
\begin{equation}\label{ouyang}
\mathbb{E} \big[ f(\bar{\m{x}}_k) + g(\bar{\m{x}}_k) - f(\m{x}^*)
- g(\m{x}^*) + \|A\bar{\m{x}}_k + B\bar{\m{y}}_k - \m{b}\| \big]
\le c/\sqrt{k},
\end{equation}
where the bar over an iterate means the average of the first $k$ iterates.
Without some additional information, such as
$f(\bar{\m{x}}_k) + g(\bar{\m{x}}_k) \ge f(\m{x}^*) + g(\m{x}^*)$,
this bound is not strong enough to ensure that the expected
objective value or constraint violation tend to zero.
In \cite{Suzuki13} the same algorithm is considered, but in the special
case that $B = -\m{I}$, where $\m{I}$ is the identity, and
$\C{A}\m{x} \in \C{Y}$ for all $\m{x} \in \C{X}$.
For this special case, $\bar{\m{y}}^k$ can be replaced by $A\bar{\m{x}}^k$
to obtain a feasible point, and (\ref{ouyang}) yields an $\C{O}(1/\sqrt{k})$
bound for the objective error.
In \cite{AzadiSra14} the error bound (\ref{ouyang}) is sharpened to
$\C{O}(1/k)$ by further developing the algorithm in \cite{Ouyang13}
by introducing a more complex averaging process and additional assumptions
such as both $\C{X}$ and $\C{Y}$ compact, and the dual multipliers are bounded.
Another variation of the method in \cite{Ouyang13} is given in
\cite{ZhongKwok14} with an error bound of $\C{O}(1/k)$.

The paper \cite{ZhaoLiZhou15} seems to be the first to realize the
potential benefit of solving the $\m{x}$-subproblem with greater accuracy.
Using $M_k = \C{O}(k^{2\varrho})$ inner iterations for the
$\m{x}$-subproblem with $\varrho > 1$,
an $\C{O}(1/k)$ bound was established for the left side
of (\ref{ouyang}).
The paper \cite{LiuShangCheng17} seems to represent the current
state-of-the-art for problems of the form (\ref{P}) with smooth $f_j$
and potentially nonsmooth $g$.
There were two fundamental innovations.
First, for their algorithm ASVRG-ADMM,
the $\m{x}$-subproblem takes advantage of both
a momentum acceleration trick from \cite{Tseng10} and variance reduction
techniques from \cite{Johnson2013} when performing a
fixed number $m$ inner iterations with
a fixed batch size for the stochastic gradients.
Second, in the analysis of ASVRG-ADMM, the authors exploit an observation
from \cite{ZhengKwok16} to obtain an $\C{O}(1/k)$ bound for both the
objective error and constraint violation.
%

In comparing AS-ADMM to the previous work,
the STOC-ADMM scheme proposed in \cite{Ouyang13}, and the various
modifications of it, use one stochastic gradient step in each ADMM
iteration to approximately solve
the $\m{x}$-subproblem, while the
scheme ASVRG-ADMM proposed in \cite{LiuShangCheng17} uses
a fixed number $m$ inner iterations.
In contrast, our AS-ADMM uses an adaptive $M_k$
accelerated stochastic gradient iterations to solve the $\m{x}$-subproblem
with increasing accuracy as the iterations progress.
We found this strategy particularly effective in our earlier work
\cite{HagerZhang19} on an inexact, adaptive ADMM scheme.
The number of iterations is chosen so
as to achieve a convergence rate of either $\C{O}(1/k)$ or
$\C{O}(k^{-1} \log k)$, based on the theory in our paper.
In a specific adaptive scheme that we analyze,
$M_k = \C{O}(k^\varrho)$ with $\varrho \ge 1$.

In the ASVRG-ADMM scheme, the $\m{y}$-subproblem is solved at each
inner iteration; hence, in $k$ iterations, ASVRG-ADMM will solve the
$\m{y}$-subproblem $mk$ times.
In contrast, AS-ADMM
treats the $\m{y}$-subproblem as a single step in the outer iteration,
and it is only solved $k$ times during $k$ iterations.

Another fundamental difference between these schemes is that AS-ADMM
does not require an estimate for the Lipschitz constant of $\nabla f$,
while ASVRG-ADMM uses the Lipschitz constant within the algorithm,
as is typical in stochastic gradient techniques.
In ASVRG-ADMM the Lipschitz constant is used to compute the momentum
parameter which appears within the steps of the algorithm.
Hence, a poor estimate of the Lipschitz constant could significantly
effect the performance of ASVRG-ADMM and other stochastic ADMMs.
If a good estimate of either the local or global Lipschitz constant were known,
then it can be exploited in AS-ADMM, but it is not required in the algorithm.

A fundamental difference between the stochastic and
deterministic ADMM literature is that in the deterministic setting,
the literature typically establishes convergence of the iterates to
a stationary point for (\ref{P}), assuming the gradient of $f$ is Lipschitz
continuous.
The corresponding convergence results in the stochastic setting have not
yet been established; what is established is the convergence of the
expected objective error and constraint violation.
However, under strengthened assumptions, such as strong convexity,
convergence of the expected ergodic error as well as convergence of
the expected iterate error can be deduced (see Appendix).

A very recent paper \cite{XieShanbhag20} developed an inexact stochastic
gradient algorithm SI-ADMM for a different version of (\ref{P}), where not only
$f$ is viewed as stochastic, but also $g$.
To incorporate the setting of \cite{XieShanbhag20}
in (\ref{P}), one should also view $g$ as the sum of
component functions $g_j$, just like $f$.
The algorithm in \cite{XieShanbhag20} differs from our algorithm in that
SI-ADMM is based on gradient steps for the augmented Lagrangian and
proximal term, while AS-ADMM employs a linearization technique
described in item 3 below.
The assumptions in \cite{XieShanbhag20} imply that both
$f_j$ and $g_j$ are strongly convex and Lipschitz continuous,
that $\C{X} = \mathbb{R}^{n_1}$ and $\C{Y} = \mathbb{R}^{n_2}$, and
that the linear constraint in (\ref{P}) has full row rank.
In this very smooth and strongly convex setting,
a linear convergence rate for the expected error in the SI-ADMM iterates
is established.
In the Appendix of our paper, we also show that the expected error in the
AS-ADMM iterates converges to zero at a linear rate
when $f$ and $g$ are strongly convex.

In more detail, some features of AS-ADMM are the following:
\begin{enumerate}
\item
The memory cost of AS-ADMM is low since
the prior stochastic gradients and iterates are not saved,
which is advantageous in big data applications.
For a specific choice of $\eta_k$ and $M_k$ given in
(\ref{Sec33-12ba2}), we show in Theorem~\ref{coll-2} that the expectation of the
objective error and constraint
violation for an ergodic mean of the AS-ADMM iterates is $\C{O}(1/k)$.
The Appendix introduces additional assumptions to obtain results
concerning the convergence of the expected error in the iterates.
For example, when $\C{M}_k - \beta A\tr A$ is uniformly positive definite,
then the iterates are bounded in expectation, and when $f$ and $g$ are
strongly convex, the expected error in the iterates converges linearly to zero.

\item
Although the AS-ADMM algorithm does not require knowledge of the Lipschitz
constant for the gradient of $f$, faster convergence
may be possible when a good estimate of the Lipschitz constant $\nu$
for $\nabla f_j$, $1 \le j \le N$, is known and exploited.
In particular, the convergence results apply when $\eta_k$ reaches the
interval $(0, 1/\nu)$;
for the choice of $\eta_k$ given in (\ref{Sec33-12ba2}), $\eta_k$ tends to
zero, so it eventually lies in the interval where convergence is guaranteed.
But if the Lipschitz constant is known, we could always take 
$\eta_k \in (0, 1/\nu)$ and the convergence rates would be valid from the
start of the iterations.
\item
The routine {\tt xsub} is obtained from the deterministic inexact ADMM scheme
in \cite{HagerZhang19b} by replacing the full gradient by a
stochastic gradient.
In the deterministic setting, it is shown in \cite{HagerZhang19b}
(see Lemma 3.1 and the parameter choice (2.4) in \cite{HagerZhang19b})
that this inexact ADMM is an accelerated scheme for solving the problem
\begin{equation}\label{real-xsub}
\arg\min\limits_{\m{x}\in\mathcal{X}}
\mathcal{L}_\beta\left(\m{x},\m{y}^k,\LAMBDA^k\right)
+ \frac{1}{2}\|\m{x} - \m{x}^k\|_{\C{D}_k}^2, \quad
\C{D}_k = \C{M}_k - \beta A \tr A.
\end{equation}

Note that both the objective function and the penalty term of
(\ref{real-xsub}) are linearized to some degree in the optimization
problem contained in {\tt xsub}.
The objective function of (\ref{real-xsub}) is linearized by replacing
the objective $f$ by $\nabla f_j$ for some $j$, while the penalty term is
partly linearized by including a proximal term of the form
$(1/2)\|\m{x} - \m{x}^k\|_{\C{M}_k - \beta A \tr A}^2$ in (\ref{real-xsub}).
This proximal term annihilates $(\beta/2)\|A\m{x}\|^2$ in the penalty term.
If $\C{M}_k$ and $\C{H}$ in {\tt xsub} were a multiple of the identity,
then the Hessian of the objective for the optimization problem in
{\tt xsub} would be a multiple of the identity.
The constraint $\C{M}_k - \beta A\tr A \succeq \m{0}$ in AS-ADMM arises from
the proximal term in (\ref{real-xsub}).
\item
AS-ADMM allows for variance reduction techniques.
In each iteration of {\tt xsub}, a stochastic
gradient $\hat{\m{g}}_t$ of the function $f$ at $\hat{\m{x}}_t$ is generated,
and the user has the flexibility of choosing a zero mean random vector
$\m{e}_t$ to reduce the variance of $\hat{\m{g}}_t$.
A trivial choice is $\m{e}_t = \m{0}$; however, faster convergence is
observed in the numerical experiments when a variance reduction technique
is employed.
\item
In the standard deterministic Gauss-Seidel version of ADMM, a dual step
$s \in (0, (1+\sqrt{5})/2)$ (the open interval) is used.
In the stochastic AS-ADMM, the stepsize constraint is
$s \in (0, (1+\sqrt{5})/2]$ (the half-open interval)
since we only show convergence of the function values.
If $M_k=1$ and $N = 1$,
then AS-ADMM becomes the standard linearized ADMM.
If $M_k > 1$ and $N = 1$,
then AS-ADMM is a deterministic inexact ADMM,
where the  $\m{x}$-subproblems of ADMM are solved inexactly using $M_k$
accelerated gradient iterations.
Hence, our convergence results for AS-ADMM also imply convergence results
for an inexact deterministic ADMM based on
$M_k$ accelerated gradient iterations.
Similar to the Gauss-Seidel version of ADMM,
$s \in (0, (1+\sqrt{5})/2)$ guarantees convergence of the iterates
for this inexact deterministic ADMM, a result not previously known in the
literature.
In fact, the more general multi-block convergence results in
\cite{HagerZhang19, HagerZhang19b} require that $s \in (0, 1)$.
\item
As shown in the analysis, the constraint in AS-ADMM that
$\C{D}_k = \C{M}_k - \beta A \tr A$
is positive semidefinite can be weakened to
$(\m{x}^{k+1} - \m{x}^k)\tr \C{D}_k (\m{x}^{k+1} - \m{x}^k)\ge 0$ for all $k$
sufficiently large.
In Remark~\ref{remxx}, we show that when $\C{M}_k = \rho_k \m{I}$,
there is an easy and effective way to adjust $\rho_k$ during the iterations,
based on an underestimate of the largest eigenvalue of $\beta A \tr A$,
so as to satisfy the weakened constraint on $\C{D}_k$ when $k$ is
sufficiently large.
\item
Our numerical experiments show that AS-ADMM performs much better
than deterministic ADMM methods for solving problem (\ref{P})
when it is expensive to compute the exact gradient of $f$,
and it is competitive or faster than other
state-of-the-art stochastic ADMM type algorithms
\cite{GM75,LiuShangCheng17,Ouyang13,OuyangLan15},
especially when the linear constraints are not simple.
\end{enumerate}

The paper is organized as follows.
Section~\ref{Sepre} introduces some notation and assumptions.
Detailed convergence analysis of AS-ADMM is given in Sections~\ref{VI} and
\ref{sec322-key}.
Incremental sampling strategies
and a linearized variant of AS-ADMM are also briefly discussed in
Sections~\ref{extension:incremental} and \ref{extension:linearized}.
Numerical experiments comparing AS-ADMM with both deterministic and
stochastic ADMM type algorithms are given in Section~\ref{sec-experiments}.
The Appendix develops properties for the expected iterates under stronger
assumptions.
In particular, the AS-ADMM iterates are bounded in expectation when
the proximal term is uniformly positive definite,
while the expected error in the iterates converges to zero at a linear
rate under a strong convexity assumption.

\section{Notation and Assumptions}\label{Sepre}
Let  $\mathcal{R}$, $\mathcal{R}^n$, and $\mathcal{R}^{n\times m}$
be the sets of  real numbers,  $n$ dimensional real column vectors,
and $n\times m$ real matrices, respectively.
Let $\m{I}$ denote the identity matrix and $\m{0}$ denote zero matrix/vector.
For symmetric matrices $A$ and $B$ of the same dimension, $A \succ B$
($A \succeq B$) means $A - B$ is a positive definite (semidefinite) matrix.
For any symmetric matrix $G$, $\|\m{x}\|_G^2 := \m{x} \tr G \m{x}$,
where the superscript $\tr$ denotes the transpose.
Note that $G$ could be indefinite with $\m{x} \tr G \m{x} < 0$ for some $\m{x}$.
If $G$ is positive definite, then $\|\m{x}\|_G$ is a norm.
We use $\|\cdot\|$ and $\langle \cdot,\cdot\rangle$
to denote the standard Euclidean norm and inner product;
$\nabla f(x)$ is the gradient of $f$ at $\m{x}$.
For convenience in the analysis, we define
\begin{equation}\label{vector-w-v}
\m{w}=\left(\begin{array}{c}
 \m{x}\\\m{y}\\ \LAMBDA
\end{array}\right)\quad \mbox{and} \quad
\mathcal{J}(\m{w})=\left(\begin{array}{c}
  -A\tr \LAMBDA\\ -B\tr\LAMBDA\\  A\m{x}+B\m{y}-\m{b}
\end{array}\right).
\end{equation}
We also define $F(\m{w}) = f(\m{x}) + g(\m{y})$ and
$\m{w}^k = [\m{x}^k \;,\; \m{y}^k \;,\; \g{\lambda}^k]$.
The affine map $\mathcal{J}(\cdot)$ is skew-symmetric in the sense that
\begin{eqnarray} \label{Sec1-J}
\left(\m{w}-\m{v}\right)\tr
\left[\mathcal{J}(\m{w})-\mathcal{J}(\m{v})\right]=0
\end{eqnarray}
for all $\m{v}$ and
$\m{w}\in \mathbb{R}^{n_1} \times \mathbb{R}^{n_2} \times \mathbb{R}^n$.
In other words, the matrix associated with $\C{J}$ is skew symmetric.

The point $\m{w}^*:=(\m{x}^*,\m{y}^*,\LAMBDA^*)\in \Omega :=$
$\mathcal{X}\times\mathcal{Y}\times \mathcal{R}^n$
is a saddle-point of the Lagrangian $\mathcal{L}$, given in (\ref{Lagr}), if
\[
\mathcal{L}\left(\m{x}^*,\m{y}^*,\LAMBDA\right)
\leq\mathcal{L}\left(\m{x}^*,\m{y}^*,\LAMBDA^*\right)
\leq \mathcal{L}\left(\m{x},\m{y},\LAMBDA^*\right)
\]
for every $\m{w} = (\m{x}, \m{y}, \LAMBDA) \in \Omega$.
It follows that
\[
\begin{array}{lllll}
f(\m{x})- f(\m{x}^*) &+&
(\m{x}-\m{x}^*)\tr \left(-A\tr \LAMBDA^*\right)&\geq&  0, \\
 g(\m{y})- g(\m{y}^*) &+&
(\m{y}-\m{y}^*)\tr \left(- B\tr\LAMBDA^*\right)&\geq& 0, \\
&&A\m{x}^*+B\m{y}^*- \m{b} &=&  \m{0}.
\end{array}
\]
These inequalities are equivalent to the variational inequality
\begin{equation}\label{Sec3-1}
\quad F(\m{w})- F(\m{w}^*) +(\m{w} -\m{w}^*)\tr \mathcal{J}(\m{w}^*) \geq  0
\end{equation}
for all $\m{w}\in \Omega$. 
Note that $\m{w}^*$ satisfies (\ref{Sec3-1}) if and only if
$\m{w}^*$ is a primal-dual solution of problem (\ref{P}).
Let $\C{W}^*$ denote the set
of $\m{w}^* \in \Omega$ satisfying (\ref{Sec3-1}).

Throughout the paper, we make the following assumptions:
\smallskip
\begin{itemize}
\item[(a1)]
The primal-dual solution set $\C{W}^*$ of the problem $(\ref{P})$ is nonempty.
\item[(a2)]
The problem
\[
\min \left\{ g(\m{y}) + (\beta/2) \m{y}\tr \m{B}\tr\m{B}\m{y} +
\m{z}\tr\m{y} : \m{y} \in \C{Y} \right\}
\]
has a minimizer for any $\m{z} \in \mathbb{R}^{n_2}$.
\item[(a3)]
For some $\nu > 0$ and $\C{H} \succ \m{0}$,
the gradients $\nabla f_j$ satisfy the Lipschitz condition
\begin{equation}\label{g-lipschitz}
\| \nabla f_j (\m{x}_1) - \nabla f_j (\m{x}_2)
\|_{\C{H}^{-1}} \le \nu \|\m{x}_1 - \m{x}_2\|_{\C{H}}
\end{equation}
for every $\m{x}_1, \m{x}_2 \in \C{X}$ and $j = 1, 2, \ldots, N$.
\end{itemize}
\smallskip

By a Taylor expansion, (a3) implies that $f$ is
$\nu$-bounded in the following sense:
\begin{equation}\label{nv-smooth}
f(\m{x}_1)\leq f(\m{x}_2)+\langle \nabla f(\m{x}_2), \m{x}_1 - \m{x}_2 \rangle +
\frac{\nu}{2}\| \m{x}_1-\m{x}_2\|_{\C{H}}^2
\end{equation}
for every $\m{x}_1, \m{x}_2 \in \C{X}$.

\section{Variational Characterization\label{VI}}
The following lemma shows a key recursive property of 
the iterates $\{\m{x}_t \}$ generated by {\tt xsub}.
Note that $\phi_k$ below is the objective function for (\ref{real-xsub}),
which {\tt xsub} is minimizing.
\begin{lemma} \label{Sec31-302}
Let us define $\Gamma_t= 2/(t(t+1))$ and
\begin{equation}\label{phi-psi-def}
\phi_k(\m{x})=
f(\m{x}) + \psi_k( \m{x}), \quad \mbox{ where } \quad
\psi_k(\m{x}) =
\frac{1}{2}\left\|\m{x}-\m{x}^k\right\|_{\C{M}_k}^2 +
\langle \m{h}^k,\m{x}\rangle,
\end{equation}
and $\m{h}^k =$
$-A\tr \left[\LAMBDA^k-\beta(A\m{x}^k+B\m{y}^k-\m{b})\right]$.
Then, for any $\m{x} \in \mathcal{X}$ and $k$ with $\eta_k \in (0, 1/\nu)$,
we have
\begin{equation} \label{1-sec1-5}
\frac{1}{\Gamma_t}\left[\phi_k({\m{x}}_{t+1})
-\phi_k(\m{x})\right]\leq\left \{\begin{array}{lll}
\theta_1, &&t=1,\\
\frac{1}{\Gamma_{t-1}}\left[\phi_k({\m{x}}_{t})
- \phi_k(\m{x})\right]+\theta_t,&& t\geq 2,
\end{array}\right.
\end{equation}
where
\begin{eqnarray}
\theta_t &=&
\frac{1}{\eta_k}\left[\left\|\m{x}-\breve{\m{x}}_{t}\right\|_{\C{H}}^2
-\left\|\m{x}
-\breve{\m{x}}_{t+1}\right\|_{\C{H}}^2\right]
-\frac{t}{2}\left\|\m{x}-\breve{\m{x}}_{t+1}\right\|_{\C{M}_k}^2
\label{theta_t}\\
&& \quad \quad +t\langle \DELTA_t, \breve{\m{x}}_{t}-\m{x}\rangle
+\frac{\eta_k t^2}{4}\frac{\left\|\DELTA_t\right\|_{\C{H}^{-1}}^2}
{(1-\eta_k\nu)},\quad t\geq 1,\;\; and \nonumber  \\
\DELTA_t&=&\nabla f(\hat{\m{x}}_{t})-\m{d}_t. \label{1-sec1-05}
\end{eqnarray}
\end{lemma}
\begin{proof}
By the updates of ${\m{x}}_{t+1}$ and $\hat{\m{x}}_{t}$, we have
\begin{equation}\label{*}
\beta_t(\breve{\m{x}}_{t+1}-\hat{\m{x}}_{t})
+(1-\beta_t)({\m{x}}_{t}-\hat{\m{x}}_{t}) =
{\m{x}}_{t+1}- \hat{\m{x}}_{t}=
\beta_t \m{s}_t , \quad
\m{s}_t = \breve{\m{x}}_{t+1} - \breve{\m{x}}_t.
\end{equation}
Since $f$ is $\nu$-bounded (\ref{nv-smooth}), the following relations hold
due to  (\ref{*}) and the convexity of $f$:
\begin{eqnarray}\label{1-sec1-3}
f({\m{x}}_{t+1})&\leq&
f(\hat{\m{x}}_{t})+
\left\langle \nabla f(\hat{\m{x}}_{t}),{\m{x}}_{t+1}
-\hat{\m{x}}_{t}\right\rangle
+\frac{\nu}{2}\left\|{\m{x}}_{t+1}
-\hat{\m{x}}_{t}\right\|_{\C{H}}^2\nonumber\\
&=& f(\hat{\m{x}}_{t})
+\left\langle \nabla f(\hat{\m{x}}_{t}),
\beta_t(\breve{\m{x}}_{t+1}
-\hat{\m{x}}_{t})+(1-\beta_t)
({\m{x}}_{t}-\hat{\m{x}}_{t})\right\rangle
+\frac{\nu\beta_t^2}{2}\left\|\m{s}_t\right\|_{\C{H}}^2\nonumber\\
&=&(1-\beta_t)\left[f(\hat{\m{x}}_{t})
+\left\langle \nabla f(\hat{\m{x}}_{t}),{\m{x}}_{t}
-\hat{\m{x}}_{t}\right\rangle\right]+\beta_t R_f
+\frac{\nu\beta_t^2}{2}\left\|\m{s}_t\right\|_{\C{H}}^2\nonumber\\
&\leq &(1-\beta_t)f({\m{x}}_{t})+  \beta_t R_f
+\frac{\nu\beta_t^2}{2}\left\|\m{s}_t\right\|_{\C{H}}^2,
 \end{eqnarray}
where $R_f=
f(\hat{\m{x}}_{t})+
\left\langle \nabla f(\hat{\m{x}}_{t}),\breve{\m{x}}_{t+1}
-\hat{\m{x}}_{t}\right\rangle$.
For any $\m{x} \in \C{X}$, it again follows from the convexity of $f$ that
\begin{eqnarray}\label{1-sec1-4}
R_f&=& f(\hat{\m{x}}_{t})
+\left\langle \nabla f(\hat{\m{x}}_{t}),\m{x}
-\hat{\m{x}}_{t}\right\rangle
+ \left\langle \nabla f(\hat{\m{x}}_{t}),\breve{\m{x}}_{t+1}
-\m{x}\right\rangle\nonumber\\
&\leq&  f(\m{x}) +
\left\langle \nabla f(\hat{\m{x}}_{t}),\breve{\m{x}}_{t+1}
-\m{x}\right\rangle.
 \end{eqnarray}

By the update formula ${\m{x}}_{t+1} =$
$\beta_t \breve{\m{x}}_{t+1} + (1-\beta_t) {\m{x}}_t$
and the convexity of $\psi_k$, we have
\begin{equation}\label{psi-convex}
\psi_k({\m{x}}_{t+1}) \le
\beta_t \psi_k(\breve{\m{x}}_{t+1} )
+(1-\beta_t) \psi_k({\m{x}}_{t}).
\end{equation}
Combine (\ref{1-sec1-3}), (\ref{1-sec1-4}), and (\ref{psi-convex})
with the definition of $\phi_k(\m{x})$ { in (\ref{phi-psi-def})},
to obtain
\begin{eqnarray}
\phi_k({\m{x}}_{t+1})&\leq&
(1-\beta_t)f({\m{x}}_{t})
+  \beta_t\left[ f(\m{x}) +
\left\langle \nabla f(\hat{\m{x}}_{t}),\breve{\m{x}}_{t+1}
-\m{x}\right\rangle\right]
+\frac{\nu\beta_t^2}{2}\left\|\m{s}_t\right\|_{\C{H}}^2
+  \psi_k({\m{x}}_{t+1}) \nonumber\\
&\leq&  (1-\beta_t)\phi_k({\m{x}}_{t})
+ \beta_t\left[ f(\m{x}) +
\left\langle \nabla f(\hat{\m{x}}_{t}),\breve{\m{x}}_{t+1}
-\m{x}\right\rangle\right]
+\frac{\nu\beta_t^2}{2}\|\m{s}_t\|_{\C{H}}^2
+ \beta_t \psi_k(\breve{\m{x}}_{t+1}).
\label{1-sec1-2}
\end{eqnarray}

In {\tt xsub} of AS-ADMM,
\[
\breve{\m{x}}_{t+1}
 = \arg\min\limits_{\m{x}\in \mathcal{X}} \; H(\m{x}) :=
\langle \m{d}_t,\m{x}\rangle
+\frac{\gamma_t}{2}\left\|\m{x}
-\breve{\m{x}}_{t}\right\|_{\C{H}}^2+ \psi_k(\m{x}),
\]
where $\psi_k$ is defined in (\ref{phi-psi-def}).
Since $H$ is a quadratic with $\nabla^2 H = \gamma_t \C{H} + \C{M}_k$,
we have
\[
H(\m{x}) = H(\breve{\m{x}}_{t+1}) + \nabla H(\breve{\m{x}}_{t+1})
(\m{x} - \breve{\m{x}}_{t+1})
+ \frac{1}{2}\|\m{x} - \breve{\m{x}}_{t+1}\|_{\gamma_t\C{H} + \C{M}_k}^{ 2} .
\]
By the first-order optimality condition, we have
$\nabla H(\breve{\m{x}}_{t+1}) (\m{x} - \breve{\m{x}}_{t+1}) \ge 0$
for all { $\m{x} \in \C{X}$},
which implies that $H(\m{x}) \ge$ $H(\breve{\m{x}}_{t+1})
+ 0.5\|\m{x} - \breve{\m{x}}_{t+1}\|_{\gamma_t\C{H} + \C{M}_k}^{{2}}$
for all $\m{x} \in \C{X}$.
Rearrange this inequality to obtain
\begin{eqnarray}
&\langle \m{d}_t,\breve{\m{x}}_{t+1}-\m{x}\rangle+
\frac{\gamma_t}{2}\left\| {\m{s}}_{t}\right\|_{\C{H}}^2
+ \psi_k(\breve{\m{x}}_{t+1}) \le& \nonumber \\
&\frac{\gamma_t}{2}\left\|\m{x}-\breve{\m{x}}_{t}\right\|_{\C{H}}^2
+ \psi_k(\m{x}) -\frac{1}{2}\|\m{x}
-\breve{\m{x}}_{t+1}\|_{\gamma_t \C{H} + \C{M}_k}^2.&
\label{1-sec1-6}
\end{eqnarray}
Substituting
$\nabla f(\hat{\m{x}}_{t})= \DELTA_t+\m{d}_t$ in (\ref{1-sec1-2})
and utilizing (\ref{1-sec1-6}) yields
\begin{eqnarray}\label{1-sec1-7}
 \phi_k({\m{x}}_{t+1})
&\leq&  \beta_t\left[f(\m{x})
+\langle \m{d}_t,\breve{\m{x}}_{t+1}-\m{x}\rangle+
\frac{\gamma_t}{2}\left\|\m{s}_{t}\right\|_{\C{H}}^2
+\psi_k(\breve{\m{x}}_{t+1})\right] \nonumber\\
&&+(1-\beta_t) \phi_k({\m{x}}_{t})
+\beta_t\langle \DELTA_t,\breve{\m{x}}_{t+1}-\m{x}\rangle
+\frac{\nu\beta_t^2}{2}\|s_t\|_{\C{H}}^2
-\frac{\beta_t\gamma_t}{2}\left\|
\m{s}_{t}\right\|_{\C{H}}^2 \nonumber\\
&\leq&  \beta_t\left[f(\m{x})+\psi_k(\m{x}) +
\frac{\gamma_t}{2}\left\|\m{x}-\breve{\m{x}}_{t}\right\|_{\C{H}}^2
-\frac{1}{2}\|\m{x} -\breve{\m{x}}_{t+1}\|_{\gamma_t \C{H}
+\C{M}_k}^2\right]\nonumber\\
&& +(1-\beta_t) \phi_k({\m{x}}_{t}) +R_d \nonumber \\
&=&
\beta_t\left[\phi_k(\m{x}) +
\frac{\gamma_t}{2}\left\|\m{x}-\breve{\m{x}}_{t}\right\|_{\C{H}}^2
-\frac{1}{2}\|\m{x} -\breve{\m{x}}_{t+1}\|_{\gamma_t \C{H} +\C{M}_k}^2\right]
+(1-\beta_t) \phi_k({\m{x}}_{t}) +R_d
 \end{eqnarray}
where
\begin{eqnarray}\label{1-sec1-8}
R_d&=& \beta_t\langle \DELTA_t,\breve{\m{x}}_{t+1}-\m{x}
-\breve{\m{x}}_{t}+\breve{\m{x}}_{t}\rangle
+\frac{\nu\beta_t^2}{2}\left\|s_t\right\|_{\C{H}}^2
- \frac{\beta_t\gamma_t}{2}\left\|\breve{\m{x}}_{t+1}
-\breve{\m{x}}_{t}\right\|_{\C{H}}^2\nonumber\\
&=&  \beta_t\langle \DELTA_t,\breve{\m{x}}_{t}
-\m{x}\rangle +\beta_t\langle \DELTA_t,\m{s}_t\rangle
-\frac{\beta_t\gamma_t
-\nu\beta_t^2}{2}\left\|\m{s}_t\right\|_{\C{H}}^2 \nonumber\\
&\leq &\beta_t\langle \DELTA_t,\breve{\m{x}}_{t}
-\m{x}\rangle
+\beta_t \left\|\DELTA_t\right\|_{\C{H}^{-1}}\left\|s_t\right\|_{\C{H}}
-\frac{\beta_t\gamma_t-\nu\beta_t^2}{2}\left\|\m{s}_t\right\|_{\C{H}}^2
\nonumber\\
&=& \beta_t\langle \DELTA_t,\breve{\m{x}}_{t}-\m{x}\rangle
+\beta_t\gamma_t\left[ \frac{1}{\gamma_t}
\left\|\DELTA_t\right\|_{\C{H}^{-1}}
\left\|\m{s}_t\right\|_{\C{H}}
-\frac{1-\nu\beta_t/\gamma_t}{2}\left\|\m{s}_t\right\|_{\C{H}}^2\right].
 \end{eqnarray}
By the choice for $\beta_t$ and $\gamma_t$, we have
\begin{equation}\label{1-sec1-9}
1-\frac{\nu\beta_t}{\gamma_t}=
1-\nu\frac{2}{t+1}\frac{t\eta_k}{2}=1-\frac{t}{t+1}\eta_k\nu>1-\eta_k\nu> 0.
\end{equation}
For $c > 0$, use the inequality
\[
0 \le
\left( \frac{a}{2\sqrt{c}}\|\m{x}\|_{\C{H}^{-1}}-
\sqrt{c}\|\m{y}\|_{\C{H}}\right)^2 =
\frac{a^2}{4c}\|\m{x}\|_{\C{H}^{-1}}^2+
c \|\m{y}\|_{\C{H}}^2-a\|\m{x}\|_{\C{H}^{-1}} \|\m{y}\|_{\C{H}}
\]
to obtain
\begin{equation}\label{bill33}
a\left\|\DELTA_t\right\|_{\C{H}^{-1}} \left\|\m{s}_t\right\|_{\C{H}}
-c \left\|\m{s}_t\right\|_{\C{H}}^2\leq\frac{a^2}{4c}
\left\|\DELTA_t\right\|_{\C{H}^{-1}}^2.
\end{equation}
Note that $c=[1-\nu\beta_t/\gamma_t]/2 > 0$ by (\ref{1-sec1-9}).
Insert this choice for $c$ and $a=1/\gamma_t$ in (\ref{bill33}),
and use the resulting inequality in (\ref{1-sec1-8}) to obtain
\[
R_d \le
\beta_t\langle \DELTA_t,\breve{\m{x}}_{t}-\m{x}\rangle +
\frac{\beta_t}{2(\gamma_t - \nu \beta_t)} \|\DELTA_t\|_{\C{H}^{-1}}^2 \le
\beta_t\langle \DELTA_t,\breve{\m{x}}_{t}-\m{x}\rangle +
\frac{\beta_t}{2\gamma_t(1 - \nu \eta_k)} \|\DELTA_t\|_{\C{H}^{-1}}^2,
\]
where the last inequality is due to (\ref{1-sec1-9}).
Combining this inequality with (\ref{1-sec1-7}) gives
\begin{eqnarray}
\phi_k({\m{x}}_{t+1})
&\leq& (1-\beta_t) \phi_k({\m{x}}_{t})
+\beta_t\phi_k(\m{x})
+\frac{\beta_t\gamma_t}{2}\left[ \left\|\m{x}
-\breve{\m{x}}_{t}\right\|_{\C{H}}^2-\left\|\m{x}
-\breve{\m{x}}_{t+1}\right\|_{\C{H}}^2\right]\nonumber\\
&&-\frac{\beta_t }{2}\left\|\m{x}
-\breve{\m{x}}_{t+1}\right\|_{\C{M}_k}^2
+\beta_t\langle \DELTA_t,\breve{\m{x}}_{t}-\m{x}\rangle
+\frac{\beta_t}{2\gamma_t}\frac{\|\DELTA_t\|_{\C{H}^{-1}}^2}{1-\nu\eta_k}.
\nonumber
\end{eqnarray}
{ Now, by subtracting $\phi_k(\m{x})$ from each side of the above inequality, we obtain}
\begin{eqnarray}\label{1-sec1-10}
\phi_k({\m{x}}_{t+1})-\phi_k(\m{x})
&\leq& (1-\beta_t)[ \phi_k({\m{x}}_{t})
-\phi_k(\m{x})]
+\frac{\beta_t\gamma_t}{2}\left[ \left\|\m{x}
-\breve{\m{x}}_{t}\right\|_{\C{H}}^2-\left\|\m{x}
-\breve{\m{x}}_{t+1}\right\|_{\C{H}}^2\right]\nonumber\\
&&-\frac{\beta_t }{2}\left\|\m{x}
-\breve{\m{x}}_{t+1}\right\|_{\C{M}_k}^2
+\beta_t\langle \DELTA_t,\breve{\m{x}}_{t}
-\m{x}\rangle
+\frac{\beta_t}{2\gamma_t}\frac{\|\DELTA_t\|_{\C{H}^{-1}}^2}{1-\nu\eta_k}.
 \end{eqnarray}

{ Finally, by} the definitions $\Gamma_t = \frac{2}{t(t+1)}$, $\beta_t = 2/(t+1)$,
and $\gamma_t = \frac{2}{t\eta_k}$, we have
\begin{equation}\label{1-sec1-11}
\beta_t\gamma_t=
\frac{4}{t(t+1)\eta_k}, \quad \frac{\beta_t\gamma_t}{\Gamma_t}
=\frac{2}{\eta_k}, \quad
\frac{\beta_t}{\Gamma_t}=t, \quad
\textrm{and} \quad
 \frac{\beta_t}{\Gamma_t\gamma_t}=\frac{\eta_k t^2}{2}.
\end{equation}
Dividing each side of (\ref{1-sec1-10}) by $\Gamma_t$ and exploiting
these relations, we deduce that (\ref{1-sec1-5}) holds for $t \ge 2$.
Since $\Gamma_1 = \beta_1 = 1$, it also follows from
(\ref{1-sec1-10}) that (\ref{1-sec1-5}) holds for $t = 1$.
\end{proof}

Based on Lemma \ref{Sec31-302},
we are able to give a variational characterization of the AS-ADMM iterates.

\begin{lemma} \label{Sec31-3}
Let $\C{D}_k$ and $\DELTA_t$ be as defined in $(\ref{real-xsub})$
and $(\ref{1-sec1-05})$ respectively, and suppose the
$\eta_k \in (0, 1/\nu)$.
Then the iterates generated by AS-ADMM satisfy
\begin{equation}\label{Bill-1}
\quad f(\m{x})-f(\m{x}^{k+1})
- \left\langle \m{x} - \m{x}^{k+1},A\tr\tilde{\LAMBDA}^k\right\rangle \geq
\left\langle \m{x}^{k+1}-\m{x} , \C{D}_k (\m{x}^{k+1}
-\m{x}^k) \right\rangle+\zeta^k,
\end{equation}
for all $\m{x} \in \C{X}$, where
\begin{eqnarray}
\tilde{\LAMBDA}^k&=&\LAMBDA^k-\beta \left( A \m{x}^{k+1}
+B \m{y}^{k}  - \m{b} \right), \mbox{ and} \label{lambda_tilde}\\
\zeta^k &=& \frac{2}{M_k(M_k+1)}
\bigg[\frac{1}{\eta_k}\left(\left\|\m{x}
-\breve{\m{x}}^{k+1}\right\|_{\C{H}}^2-
\left\|\m{x}-\breve{\m{x}}^{k}\right\|_{\C{H}}^2\right)
\label{zeta_k} \\
&& \quad -\sum\limits_{t=1}^{M_k}t\langle \DELTA_t, \breve{\m{x}}_{t}
-\m{x}\rangle -\frac{\eta_k}{4(1-\eta_k\nu)}
\sum\limits_{t=1}^{M_k}t^2\left\|\DELTA_t\right\|_{\C{H}^{-1}}^2\bigg].
\nonumber
\end{eqnarray}
\end{lemma}
\begin{proof}
Let us define $T = M_k$.
{ Summing} (\ref{1-sec1-5}) over $1 \le t \le T$ and recalling that
$\breve{\m{x}}^k = \breve{\m{x}}_1$,
$\m{x}^{k+1} = \m{x}_{T+1}$, and $\breve{\m{x}}^{k+1} = \breve{\m{x}}_{T+1}$,
we obtain
\begin{eqnarray}\label{1-sec1-12}
&&\frac{1}{\Gamma_T}
\left[\phi_k({\m{x}}^{k+1})-\phi_k(\m{x})\right]
\le \sum_{t=1}^T \theta_t \nonumber \\
&=& \frac{1}{\eta_k}\left[ \left\|\m{x}
-\breve{\m{x}}^{k}\right\|_{\C{H}}^2
-\left\|\m{x}-\breve{\m{x}}^{k+1}\right\|_{\C{H}}^2\right]
-\frac{1}{2}\sum\limits_{t=1}^{T}t\left\|\m{x}
-\breve{\m{x}}_{t+1}\right\|_{\C{M}_k}^2\nonumber\\
&&+\sum\limits_{t=1}^{T}t\langle \DELTA_t, \breve{\m{x}}_{t}-\m{x}\rangle
+\frac{\eta_k }{4(1-\eta_k\nu)}
\sum\limits_{t=1}^{T}t^2\left\|\DELTA_t\right\|_{\C{H}^{-1}}^2
\end{eqnarray}
for any $\m{x} \in \mathcal{X}$, where $\theta_t$ is defined in (\ref{theta_t}).
Dividing the update formula
${\m{x}}_{t+1}= \beta_t \breve{\m{x}}_{t+1}+(1-\beta_t){\m{x}}_{t}$
by $\Gamma_t$ and exploiting the identity
$\beta_t/\Gamma_t = t$ from (\ref{1-sec1-11}) yields
\[
\frac{1}{\Gamma_t}{\m{x}}_{t+1}=
\frac{1}{\Gamma_{t-1}}{\m{x}}_{t}
+t \breve{\m{x}}_{t+1}.
\]
We sum over $2 \le t \le T$
and recall that $\Gamma_1=\beta_1=1$ to obtain
\begin{eqnarray}
{\m{x}}^{k+1}&=& \Gamma_T
\left\{
\frac{1}{\Gamma_1}{\m{x}}_{2} +
\sum\limits_{t=2}^{T}t\breve{\m{x}}_{t+1} \right\}
=\Gamma_T\left\{ {\m{x}}_{2}-\breve{\m{x}}_{2}+
\sum\limits_{t=1}^{T} t\breve{\m{x}}_{t+1} \right\} \nonumber \\
&=&\Gamma_T\left\{
[\beta_1\breve{\m{x}}_{2}+(1-\beta_1){\m{x}}_{1}]
- \breve{\m{x}}_{2} +
\sum\limits_{t=1}^{T} t\breve{\m{x}}_{t+1}
\right\}
=\sum\limits_{t=1}^{T}(t \Gamma_T)\breve{\m{x}}_{t+1}.
\label{h0}
\end{eqnarray}
Since $(t\Gamma_T)$ for $1 \le t \le T$ sums to 1 and
the quadratic term $\|\m{z}-\m{x}\|_{\C{M}_k}^2$ is convex in $\m{z}$,
it follows from (\ref{h0}) that for any choice of $\m{x}$, { we have}
\[
\left\|{\m{x}}^{k+1}-\m{x}\right\|_{\C{M}_k}^2\le
\sum\limits_{t=1}^{T}(t \Gamma_T)
\left\|\breve{\m{x}}_{t+1}-\m{x}\right\|_{\C{M}_k}^2.
\]
Inserting this inequality in (\ref{1-sec1-12}) gives
\begin{eqnarray}
&\frac{1}{\Gamma_T} \left[\phi_k({\m{x}}^{k+1})
-\phi_k(\m{x})+\frac{1}{2}\left\|{\m{x}}^{k+1}
-\m{x}\right\|_{\C{M}_k}^2\right]
\le \frac{1}{\eta_k}\left[ \left\|\m{x}
-\breve{\m{x}}^{k}\right\|_{\C{H}}^2-\left\|\m{x}
-\breve{\m{x}}^{k+1}\right\|_{\C{H}}^2\right] &
\nonumber\\
& +\sum\limits_{t=1}^{T}t\langle \DELTA_t, \breve{\m{x}}_{t}-\m{x}\rangle
+\frac{\eta_k }{4(1-\eta_k\nu)}
\sum\limits_{t=1}^{T}t^2\left\|\DELTA_t\right\|_{\C{H}^{-1}}^2.&
\label{1-sec1-13}
\end{eqnarray}
Now, by the definition of $\phi_k$ and $\psi_k$, we have
\begin{eqnarray*}
\phi_k({\m{x}}^{k+1})-\phi_k(\m{x})
&=& f({\m{x}}^{k+1}) -f(\m{x})
+ \psi_k({\m{x}}^{k+1}) -\psi_k(\m{x}) \quad \mbox{and}\\
\psi_k({\m{x}}^{k+1})-\psi_k(\m{x})
&=&
\left\langle \m{h}^k, {\m{x}}^{k+1}-\m{x}\right\rangle
+\frac{1}{2}\left[ \|{\m{x}}^{k+1}-\m{x}^k\|_{\C{M}_k}^2
-\|\m{x}-\m{x}^k\|_{\C{M}_k}^2 \right].
\end{eqnarray*}
By the definition of $\m{h}^k$, it follows that
\begin{eqnarray*}
 \m{h}^k&=&-A\tr \left[\LAMBDA^k
-\beta\left(A\m{x}^k+B\m{y}^k-\m{b}\right)\right]\nonumber\\
&=&-A\tr \left[\LAMBDA^k-\beta \left(A\m{x}^{k+1}+B\m{y}^k-\m{b}\right)\right]
 - \beta A \tr A  \left(\m{x}^{k+1} -\m{x}^k \right) \nonumber\\
&=& - A \tr  \tilde{\LAMBDA}^k
- \beta A \tr A  \left(\m{x}^{k+1} -\m{x}^k \right).
\end{eqnarray*}
The identity
\[
(\m{a}-\m{b})\tr \C{M}_k (\m{a}-\m{c})
=\frac{1}{2}\left\{\|\m{a}-\m{c}\|_{\C{M}_k}^2
-\|\m{c}-\m{b}\|_{\C{M}_k}^2+\|\m{a}-\m{b}\|_{ \C{M}_k}^2\right\}
\]
with $\m{a} = \m{x}^{k+1}$, $\m{b} = { \m{x}^k}$, and $\m{c} =\m{x}$
implies that
\[
\frac{1}{2}\left[ \left\|{\m{x}}^{k+1}-\m{x}^k\right\|_{\C{M}_k}^2
-\left\|\m{x}-\m{x}^k\right\|_{\C{M}_k}^2
+\left\|{\m{x}}^{k+1}-\m{x}\right\|_{\C{M}_k}^2\right] =
\left( {\m{x}}^{k+1}-\m{x}^k \right) \tr \C{M}_k
\left({\m{x}}^{k+1}-\m{x}\right).
\]
Insert all these relations in (\ref{1-sec1-13}) and make the
substitutions $T = M_k$ and { $\Gamma_T = 2/(T(T+1))$
to obtain (\ref{Bill-1}), which completes the proof.}
\end{proof}
\medskip

We now establish the following
variational inequality.
\smallskip

\begin{theorem} \label{Sec31-bz6}
If $\eta_k \in (0, 1/\nu)$, then the iterates generated by AS-ADMM satisfy
\begin{equation} \label{Bill3}
F(\m{w})-F(\tilde{\m{w}}^k) +
\left\langle \m{w} - \tilde{\m{w}}^k,
\mathcal{J}(\tilde{\m{w}}^k)\right\rangle
\geq  \langle {\m{w}} - \tilde{\m{w}}^k ,
Q_k (\m{w}^k - \m{w}^{k+1}) \rangle + \zeta^k
\end{equation}
for all $\m{w} \in \Omega$,
where $\zeta^k$ is defined in $(\ref{zeta_k})$, $\tilde{\LAMBDA}^k$ is defined in $(\ref{lambda_tilde})$, and
\begin{equation}\label{tilde-wk-Q}
\tilde{\m{w}}^k :=
\left(\begin{array}{c}
  \m{\tilde{x}}^k\\ \m{\tilde{y}}^k\\  \tilde{\LAMBDA}^k
\end{array}\right) :=
\left(\begin{array}{c}
  \m{x}^{k+1}\\ \m{y}^{k+1}\\    \tilde{\LAMBDA}^k
\end{array}\right),\quad
Q_k=\left[\begin{array}{ccccc}
           \C{D}_k && &&  \\
          &&\beta B\tr B&&    \\
      & & & &    \frac{1}{s\beta} I
\end{array}\right].
\end{equation}
\end{theorem}

\begin{proof}
Since the objective in the $\m{y}$-subproblem is the sum of a nonsmooth
and a smooth term, the first-order optimality condition can be expressed as
\begin{equation}\label{Bill1}
g(\m{y}) - g(\m{y}^{k+1})
+ \left\langle\m{y}-\m{y}^{k+1}, \m{p}_k \right\rangle\geq  0
\end{equation}
for all $\m{y}\in \mathcal{Y}$,
where $\m{p}_k$ is the gradient with respect to $\m{y}$, evaluated at
$(\m{x}^{k+1}, \m{y}^{k+1})$, of the smooth term:
\begin{eqnarray*}
\m{p}_k&=&
-B\tr  \LAMBDA^{k} +\beta B\tr\left(A\m{x}^{k+1}+B\m{y}^{k+1}-\m{b}\right)
\nonumber\\
&=& -B\tr  \left[ \LAMBDA^{k}  - \beta
\left(A\m{x}^{k+1}+B\m{y}^k-\m{b}\right)
- \beta B \left(\m{y}^{k+1} - \m{y}^k  \right)  \right] \nonumber \\
&=&   -B\tr \tilde{\LAMBDA}^{k}
+ \beta B \tr B \left(\m{y}^{k+1} - \m{y}^k  \right).
\end{eqnarray*}
Here $\tilde{\LAMBDA}^k$ is defined in $(\ref{lambda_tilde})$.
Substituting $\m{p}_k$ into (\ref{Bill1}) gives
\begin{equation}\label{Bill0}
g(\m{y})-g(\m{y}^{k+1})
- \left\langle \m{y} - \m{y}^{k+1},B\tr\tilde{\LAMBDA}^k\right\rangle \geq
\beta \langle \m{y}^{k+1}-\m{y}, {B}\tr{B} (\m{y}^{k+1} -\m{y}^k) \rangle
\end{equation}
for all $\m{y} \in \C{Y}$.

The update formula for $\g{\lambda}^{k+1}$ yields the relation
\[
A\m{x}^{k+1} + B\m{y}^{k+1} - \m{b} =
\frac{\g{\lambda}^k - \g{\lambda}^{k+1}}{s\beta} .
\]
Take the inner product {of the above equality} with $\g{\lambda}-\tilde{\g{\lambda}}^k$ to obtain
\begin{equation}\label{Bill2}
\langle \g{\lambda}-\tilde{\g{\lambda}}^k ,
A\m{x}^{k+1} + B\m{y}^{k+1} - \m{b} \rangle
= \frac{1}{s\beta} \langle \g{\lambda}-\tilde{\g{\lambda}}^k ,
\g{\lambda}^k - \g{\lambda}^{k+1} \rangle .
\end{equation}
Adding (\ref{Bill-1}), (\ref{Bill0}), and (\ref{Bill2}) yields (\ref{Bill3}).
\end{proof}

For the convergence analysis,
we need to further analyze the right side of (\ref{Bill3}).
\begin{corollary} \label{key-lemma}
If $\eta_k \in (0, 1/\nu)$, then the
iterates of AS-ADMM satisfy the following relation:
\begin{eqnarray}
& & F(\m{w})-F(\tilde{\m{w}}^k) +
( \m{w} - \tilde{\m{w}}^k)\tr \mathcal{J}(\m{w})  \nonumber \\
& \ge & \frac{1}{2}\left\{\left\|\m{w}-\m{w}^{k+1}\right\|^2_{{Q}_k}
-\left\|\m{w}-\m{w}^{k}\right\|^2_{{Q}_k}
+  \left\|\m{w}^k-\tilde{\m{w}}^k\right\|_{{G}_k}^2 \right\} + \zeta^k,
\label{bjc-39}
\end{eqnarray}
{for any $\m{w} \in \Omega$,} where $\zeta^k$ is defined in 
$(\ref{zeta_k})$, $Q_k$ is given by $(\ref{tilde-wk-Q})$ and
\begin{equation}\label{tilde-G}
{G}_k
=\left[\begin{array}{ccccc}
           \C{D}_k && &&  \\
          &&(1-s)\beta B\tr B&&  (s-1) B \tr  \\
      && (s-1) B & &    \frac{2-s}{\beta} I
\end{array}\right].
\end{equation}
\end{corollary}
\begin{proof}
The identity
\[
 2(\m{a}-\m{b})\tr  {Q}_k(\m{c}-\m{d})=
\|\m{a}-\m{d}\|_{{Q}_k}^2-\|\m{a}-\m{c}\|_{{Q}_k}^2
+ \|\m{c}-\m{b}\|_{{Q}_k}^2-\|\m{b}-\m{d}\|_{{Q}_k}^2
\]
with the choices
$\m{a} = \m{w}$, $\m{b}= \tilde{\m{w}}^k$,
$\m{c}=\m{w}^k$, and $\m{d}=\m{w}^{k+1}$ gives
\begin{eqnarray}
&\langle \m{w} - \tilde{\m{w}}^k,
Q_k (\m{w}^k - \m{w}^{k+1}) \rangle =& \label{expand} \\
&\frac{1}{2}\left\{\left\|\m{w}-\m{w}^{k+1}\right\|^2_{{Q}_k}-
\left\|\m{w}-\m{w}^{k}\right\|^2_{{Q}_k}
+ \left\|\m{w}^k-\tilde{\m{w}}^k\right\|_{{Q}_k}^2 -
\left\|\m{w}^{k+1}-\tilde{\m{w}}^k\right\|_{{Q}_k}^2\right\}.&
\nonumber
\end{eqnarray}
The update formula for $\g{\lambda}^{k+1}$, together with the definition
of $\tilde{\g{\lambda}}^k$ in (\ref{lambda_tilde}), yield the relation
\begin{equation} \label{121-zlg}
\g{\lambda}^{k+1} - \g{\lambda}^k = s\beta B(\m{y}^k - {\m{y}}^{k+1})
-s(\g{\lambda}^k - \tilde{\g{\lambda}}^k) .
\end{equation}
Hence, we have
\begin{eqnarray*}
\g{\lambda}^{k+1} - \tilde{\g{\lambda}}^k &=&
\g{\lambda}^{k+1} - \g{\lambda}^k + {\g{\lambda}}^k - \tilde{\g{\lambda}}^k \\
&=& s\beta B (\m{y}^k - \m{y}^{k+1})
+ (1-s) (\g{\lambda}^k - \tilde{\g{\lambda}}^k) .
\end{eqnarray*}
Since the only nonzero component of $\m{w}^{k+1}-\tilde{\m{w}}^k$ is
the $\g{\lambda}^{k+1} - \tilde{\g{\lambda}}^k$ component, we have
\[
\left\|\m{w}^{k+1}-\tilde{\m{w}}^k\right\|_{{Q}_k}^2 =
\frac{1}{s\beta} \left\| s\beta B (\m{y}^k - \m{y}^{k+1})
+ (1-s) (\g{\lambda}^k - \tilde{\g{\lambda}}^k) \right\|^2 .
\]
With this substitution, it follows that
\[
\left\|\m{w}^k-\tilde{\m{w}}^k\right\|_{{Q}_k}^2 -
\left\|\m{w}^{k+1}-\tilde{\m{w}}^k\right\|_{{Q}_k}^2 =
\left\|\m{w}^k-\tilde{\m{w}}^k\right\|_{{G}_k}^2 .
\]
Combine this identity with (\ref{expand}), Theorem~\ref{Sec31-bz6},
and the skew symmetry of $\C{J}$ to complete the proof.
\end{proof}
\medskip

The following theorem provides a lower bound for the $G_k$ term
in (\ref{bjc-39}).

\begin{theorem} \label{Gkterm}
The iterates of AS-ADMM satisfy
\begin{eqnarray}
\left\|\m{w}^k-\tilde{\m{w}}^k\right\|_{{G}_k}^2 &\ge &
(2-s)\beta \left\|A\m{x}^{k+1}+B\m{y}^{k+1}-\m{b}\right\|^2 + 
\left\|\m{x}^k - \m{x}^{k+1}\right\|_{\C{D}_k}^2 \nonumber \\
&&- (1-s)^2\beta \left\|A\m{x}^k +B\m{y}^k -\m{b}\right\|^2,
\label{1qaz}
\end{eqnarray}
where $G_k$ and $\C{D}_k$ are defined in 
$(\ref{tilde-G})$ and $(\ref{real-xsub})$, respectively.
\end{theorem}
\begin{proof}
By the definition of ${G}_k$ in (\ref{tilde-G}) and direct calculation, we have
\begin{eqnarray*}
\frac{1}{\beta} \left\|\m{w}^k-\tilde{\m{w}}^k\right\|_{{G}_k}^2 & = &
\frac{1}{\beta} \left\|\m{x}^k - \m{x}^{k+1}\right\|_{\C{D}_k}^2
+ (1-s) \left\| B \left( \m{y}^k- \m{y}^{k+1} \right) \right\|^2  + \\
&&  \frac{2(s-1)}{\beta} \left( \LAMBDA^k-\tilde{\LAMBDA}^k \right)
\tr B \left( \m{y}^k- \m{y}^{k+1} \right)
+ \frac{2-s}{\beta^2}  \left\| \LAMBDA^k-\tilde{\LAMBDA}^k \right\|^2.
\end{eqnarray*}
Since $\tilde{\g{\lambda}}^k - \g{\lambda}^k =
\beta(A\m{x}^{k+1} + B\m{y}^{k+1} - \m{b}) + \beta B (\m{y}^{k+1} - \m{y}^k)$,
it follows that
\begin{eqnarray} \label{G-eq}
\frac{1}{\beta } \left\|\m{w}^k-\tilde{\m{w}}^k\right\|_{{G}_k}^2 & =
& \frac{1}{\beta} \left\| \m{x}^k - \m{x}^{k+1}\right\|_{\C{D}_k}^2
 + (2-s) \left\| A\m{x}^{k+1}+B\m{y}^{k+1}-\m{b} \right\|^2 + \nonumber \\
&&  \left\| B \left( \m{y}^k- \m{y}^{k+1} \right) \right\|^2
+ 2 \left( A\m{x}^{k+1}+B\m{y}^{k+1}- \m{b} \right) \tr B
\left( \m{y}^k- \m{y}^{k+1} \right).
\end{eqnarray}
Choosing $\m{y} = \m{y}^k$ in the first-order optimality condition
(\ref{Bill1}), we have
\[
 g(\m{y}^k)- g(\m{y}^{k+1})+ \left\langle B
\left( \m{y}^k-\m{y}^{k+1} \right), - \LAMBDA^{k}
+\beta \left(A\m{x}^{k+1}+B\m{y}^{k+1}-\m{b}\right) \right\rangle\geq  0.
\]
Similarly, choosing $\m{y} = \m{y}^{k+1}$
in the first-order optimality condition (\ref{Bill1})
at the $(k-1)$-th iteration, we have
\[
 g(\m{y}^{k+1})- g(\m{y}^k)+
\left\langle B \left( \m{y}^{k+1}-\m{y}^k \right),
- \LAMBDA^{k-1} +\beta \left(A\m{x}^k+B\m{y}^k-\m{b}\right)
\right\rangle\geq  0.
\]
Adding these two inequalities and substituting
$\LAMBDA^k = \LAMBDA^{k-1} - s \beta \left(A\m{x}^k+B\m{y}^k-\m{b}\right) $,
we have
\begin{eqnarray*}
& &  \left( A\m{x}^{k+1}+B\m{y}^{k+1}- \m{b} \right) \tr B
\left( \m{y}^k- \m{y}^{k+1} \right)   \\
& \ge &  (1-s) \left( A\m{x}^k+B\m{y}^k- \m{b} \right) \tr B
\left( \m{y}^k- \m{y}^{k+1} \right) \\
& \ge& - \frac{1}{2} \left( (1-s)^2 \left\|A\m{x}^k
+B\m{y}^k -\m{b}\right\|^2 +   \left\| B \left( \m{y}^k
- \m{y}^{k+1} \right) \right\|^2  \right),
\end{eqnarray*}
where the last inequality comes from the relation
$\m{x}\tr \m{y}\geq -\frac{1}{2}\left[c\|\m{x}\|^2
+\frac{1}{c}\|\m{y}\|^2\right]$ for any $c > 0$.
Inserting this lower bound for the last term
in (\ref{G-eq}) yields (\ref{1qaz}).
\end{proof}
\section{Convergence Analysis\label{sec322-key}}
In this section, we analyze the convergence properties of AS-ADMM.
The following lemma explores how closely an ergodic average of the iterates
satisfies the first-order optimality condition (\ref{Sec3-1}).

\begin{lemma} \label{Sec3-theore1}
Suppose that for some integers $\kappa \ge 0$ and $T > 0$ and for all
$k \in [\kappa, \kappa+T]$, the following conditions are satisfied:
\begin{itemize}
\item[{\rm (A1)}]
$\C{D}_k \succeq \C{D}_{k+1} \succeq \m{0}$ and
$\mathbb{E}\big[\|\DELTA_t\|_{\C{H}^{-1}}^2\big] \le \sigma^2$
for some $\sigma > 0$,
independent of $t$ and the iteration number $k$, where $\DELTA_t$
is defined in $(\ref{1-sec1-05})$. 
\item[{\rm (A2)}]
$\eta_k \in (0, 1/2\nu]$, where $\nu > 0$ is the Lipschitz constant given
in {\rm (a3)}, and
the sequence $\eta_kM_k(M_k+1)$ is nondecreasing.
\end{itemize}
Then for every $ \m{w} \in \Omega,$ we have
\begin{eqnarray}
& \qquad \mathbb{E}\left[ F(\m{w}_T)
-F(\m{w})+(\m{w}_T-\m{w})\tr \mathcal{J}(\m{w})\right]\leq 
\frac{1}{2(1+T)}\bigg\{\sigma^2\sum\limits_{k=\kappa}^{\kappa+T} \eta_k M_k &
\label{Ex-F} \\
& + \|\m{w}-\m{w}^{\kappa}\|_{Q_{\kappa}}^2 +
\beta (1-s)^2  \left\|A\m{x}^{\kappa}+B\m{y}^{\kappa}-\m{b}\right\|^2
+ \frac{4}{M_{\kappa}(M_{\kappa}+1)\eta_{\kappa}}
\|\m{x}- \m{x}^{{\kappa}}\|_{\C{H}}^2 \bigg\} &\nonumber
\end{eqnarray}
where
\begin{equation}\label{erg-iterate}
\m{w}_T :=\frac{1}{1+T}\sum_{k={\kappa}}^{\kappa + T}\tilde{\m{w}}^{k}.
\end{equation}
\end{lemma}
\begin{proof}
Since $s \in (0, (1+\sqrt{5})/2]$
and $\beta >0$ in AS-ADMM, we have
\[
\xi_1 := \beta ((2-s) - (1-s)^2)  \ge 0 \quad \mbox{and} \quad
\xi_2 := \beta (1-s)^2 \ge 0.
\]
The inequality (\ref{1qaz}) can be rearranged into the form
\begin{eqnarray}
 \left\|\m{w}^k-\tilde{\m{w}}^k\right\|_{G_k}^2 &\ge&  \left\|\m{x}^k - \m{x}^{k+1}\right\|_{\C{D}_k}^2 +
\xi_1 \left\|A\m{x}^{k+1}+B\m{y}^{k+1}-\m{b}\right\|^2  \label{h1}  \\
&&+ \xi_2 \left( \left\|A\m{x}^{k+1}+B\m{y}^{k+1}-\m{b}\right\|^2  - \left\|A\m{x}^k +B\m{y}^k -\m{b}\right\|^2 \right). \nonumber
\end{eqnarray}
By (A1) and the fact that $s > 0$, it follows that
$Q_k$ in (\ref{tilde-wk-Q}) satisfies
$Q_k \succeq Q_{k+1} \succeq \m{0}$.
Substituting (\ref{h1}) in (\ref{bjc-39}) and utilizing the relation
$Q_k \succeq Q_{k+1}$, { we have}
\begin{eqnarray}
& & F(\tilde{\m{w}}^k) - F(\m{w}) +
( \tilde{\m{w}}^k - \m{w})\tr \mathcal{J}(\m{w}) 
\label{bjc-39-new} \\
& \le &  \frac{1}{2}\left\{\left\|\m{w}-\m{w}^{k}\right\|^2_{Q_{k}}
-\left\|\m{w}-\m{w}^{k+1}\right\|^2_{Q_{k+1}} \right\} \nonumber \\
 && + \frac{\xi_2}{2} \left\{ \left\|A\m{x}^{k}+B\m{y}^{k}-\m{b}\right\|^2
 - \left\|A\m{x}^{k+1} +B\m{y}^{k+1} -\m{b}\right\|^2 \right\} \nonumber \\
& & - \frac{1}{2}
 \left\{ \left\|\m{x}^k - \m{x}^{k+1}\right\|_{\C{D}_k}^2
+ \xi_1 \left\|A\m{x}^{k+1}+B\m{y}^{k+1}-\m{b}\right\|^2 \right\}  -\zeta^k, \nonumber 
\end{eqnarray}
where $\zeta^k$ is defined in (\ref{zeta_k}).

Sum the inequality (\ref{bjc-39-new}) over $k$ between $\kappa$ and $\kappa + T$.
Notice that 
the sum associated with the first two bracketed terms are telescoping
series while the sum associated with the third bracketed expression is
negative and can be neglected. 
Thus by the definition of $\textbf{w}_T$ in (\ref{erg-iterate}), we obtain
\begin{eqnarray}\label{sumover}
&& \sum\limits_{k=\kappa}^{\kappa+T}F(\tilde{\m{w}}^k)-(1+T)\left\{ F(\m{w})
+ \left(  \m{w}_T - \m{w} \right) \tr \mathcal{J}(\m{w})\right\} \\
&\leq& \frac{1}{2} \left\{ \|\m{w}-\m{w}^{\kappa}\|^2_{Q_{\kappa}} +
\xi_2 \left\|A\m{x}^{\kappa}+B\m{y}^{\kappa}-\m{b}\right\|^2  \right\}
-\sum\limits_{k={\kappa}}^{\kappa+T}\zeta^k.\nonumber
\end{eqnarray}
It further follows from the convexity of $F$ that
\begin{equation} \label{Sec3-bjc-6}
F(\m{w}_T)\leq \frac{1}{1+T}\sum_{k=\kappa}^{\kappa+T}F(\tilde{\m{w}}^{k}).
\end{equation}
Dividing (\ref{sumover}) by $T+1$ and utilizing (\ref{Sec3-bjc-6}), we obtain
\begin{eqnarray} \label{Sec3-bjc-5}
&& F(\m{w}_T) -F(\m{w})+ (\m{w}_T-\m{w})\tr \mathcal{J}(\m{w}) \\
&\le & \frac{1}{1+T}\left[ \frac{1}{2} \left\{ \|\m{w}-\m{w}^{\kappa}\|^2_{Q_{\kappa}}
+ \xi_2 \left\|A\m{x}^{\kappa}+B\m{y}^{\kappa}-\m{b}\right\|^2
-\sum\limits_{k={\kappa}}^{{\kappa}+T} \zeta^k \right\} \right]. \nonumber
\end{eqnarray}

Let us now focus on the $\zeta^k$ summation in (\ref{Sec3-bjc-5}).
By assumption (A2),
the sequence $\left\{M_k(M_k+1)\eta_k\right\}$ is nondecreasing for
$k \in [\kappa, \kappa + T]$;
hence, by the telescoping nature of the sum, we have
\begin{eqnarray}
&\sum\limits_{k=\kappa}^{\kappa+T}\frac{2}{M_k(M_k+1)\eta_k}
\left( \|\m{x}-\breve{\m{x}}^{k}\|_{\C{H}}^2-\|\m{x}-
\breve{\m{x}}^{k+1}\|_{\C{H}}^2\right)&
\label{Sec3-bcz-6} \\
&\le \sum\limits_{k=\kappa}^{\kappa+T}
\left(
\frac{2\|\m{x}-\breve{\m{x}}^{k}\|_{\C{H}}^2}
{M_k(M_k+1)\eta_k}
- \frac{2\|\m{x}- \breve{\m{x}}^{k+1}\|_{\C{H}}^2}
{M_{k+1}(M_{k+1}+1)\eta_{k+1}} \right)
\leq \frac{2\|\m{x}- \m{x}^{\kappa}\|_{\C{H}}^2}
{M_\kappa(M_\kappa+1)\eta_\kappa} .& \nonumber
\end{eqnarray}
For $\DELTA_t$  defined in $(\ref{1-sec1-05})$, we have
\[
\DELTA_t=\nabla f(\hat{\m{x}}_{t})-\m{d}_t=
\nabla f(\hat{\m{x}}_{t}) -\nabla f_{\xi_t}(\hat{\m{x}}_t)-\m{e}_t.
\]
Since the random variable $\xi_t \in \{1,2,\ldots, N\}$
is chosen with uniform probability and $\mathbb{E}[\m{e}_t] = \m{0}$,
it follows that $\mathbb{E}[ \DELTA_t]=\m{0}$.
Also, since $\DELTA_t$ only depends on the index $\xi_t$ while
$\breve{\m{x}}_t$ depends on $\xi_{t-1}$, $\xi_{t-2}$, $\ldots$, we have
\[
\mathbb{E} \left[ \langle \DELTA_t, \breve{\m{x}}_{t}-\m{x}\rangle \right] =
\m{0}.
\]
By (A1), we have
$\mathbb{E}(\|\DELTA_t\|_{\C{H}^{-1}}^2) \le \sigma^2$.
Since $M_k \ge 1$, it follows that
\[
\mathbb{E} \left[ \sum_{t=1}^{M_k} t^2 \|\DELTA_t\|_{\C{H}^{-1}}^2 \right] \le
\frac{\sigma^2 M_k (M_k+1)(2M_k + 1)}{6} \le
M_k^2 (M_k+1)\left( \frac{\sigma^2 }{2} \right).
\]
Combining these bounds for the terms in $\zeta^k$ defined in (\ref{zeta_k})
with the condition $\eta_k \le 1/(2\nu)$ in (A2) yields
\[
- \mathbb{E} \left[ \sum_{k=\kappa}^{\kappa+T} \zeta^k \right] \le
\frac{2\|\m{x}- \m{x}^{\kappa}\|_{\C{H}}^2}{M_\kappa(M_\kappa+1)\eta_\kappa}
+ \frac{\sigma^2}{2} \sum_{k=\kappa}^{\kappa+T} \eta_k M_k .
\]
To complete the proof, apply the expectation operator to
(\ref{Sec3-bjc-5}) and substitute this bound for the $\zeta^k$ term.
\end{proof}

Analogous to the definition (\ref{erg-iterate}), we define
\begin{equation} \label{erg-mean-xyz}
\LAMBDA_{T} =\frac{1}{1+T}\sum_{k=\kappa}^{\kappa+T}\tilde{\LAMBDA}^{k},\quad
\textbf{x}_{T} =\frac{1}{1+T}\sum_{k=\kappa}^{\kappa+T}\tilde{\textbf{x}}^{k} \quad\mbox{and} \quad
\textbf{y}_{T} =\frac{1}{1+T}\sum_{k=\kappa}^{\kappa+T}\tilde{\textbf{y}}^{k}.
\end{equation}
Lemma~\ref{Sec3-theore1} yields
a convergence result for AS-ADMM when we make
the following choice for $\eta_k$ and $M_k$ in (\ref{Ex-F}):
\begin{equation}\label{Sec33-12ba2}
\eta_k= \min \left\{ \frac{c_1}{M_k(M_k+1)}, c_2 \right\}
\quad \textrm{and} \quad
M_k=\max\left\{\lceil c_3 k^{\varrho}\rceil, M\right\},
\end{equation}
where $c_1, c_2, c_3 >0$ and
$\varrho \ge 1$ are constants, and $M>0$ is a given integer.
Choose $\kappa$ large enough that
$M_\kappa = \lceil c_3 \kappa^{\varrho}\rceil$.
As $k$ tends to infinity,
$M_k$ tends to infinity and $\eta_k$ tends to zero.
Choose $\kappa$ larger if necessary to ensure that
$\eta_\kappa = c_1/(M_\kappa(M_\kappa+1)) \le 1/(2\nu)$,
where $\nu$ is the Lipschitz constant in (A2).
Since $\eta_k M_k (M_k+1) = c_1$, a constant, and $\eta_k \in (0, 1/2\nu]$
for $k \ge \kappa$, condition (A2) of Lemma~\ref{Sec3-theore1} is satisfied
for this choice of $\kappa$.
\smallskip

\begin{theorem}\label{coll-2}
If {\rm (A1)} of Lemma~$\ref{Sec3-theore1}$ holds for all $k$
and the parameters $\eta_k$
and $M_k$ are chosen according to $(\ref{Sec33-12ba2})$,
then for every $\m{w}^* \in \C{W}^*$,  we have
\begin{equation}\label{Ergodic-rate}
\left| \mathbb{E} \big[ F(\m{w}_T)-F(\m{w}^*)\big]\right| = E_\varrho(T) =
\mathbb{E} \big[\left\|A\m{x}_{T}+B\m{y}_{T}-\m{b} \right\|\big],
\end{equation}
where $E_\varrho(T) = \C{O} (1/T)$ for $\varrho > 1$ and
$E_\varrho(T) = \C{O} (T^{-1}\log T)$ for $\varrho = 1$.
\end{theorem}

\begin{proof}
Suppose that $\kappa$ is chosen by the procedure explained beneath
(\ref{Sec33-12ba2}), which ensures that
condition (A2) of Lemma~\ref{Sec3-theore1} is satisfied for all
$k \ge \kappa$.
By assumption, (A1) holds.
Hence, the conclusion (\ref{Ex-F}) of Lemma~\ref{Sec3-theore1} holds.

First, let us analyze the left side of (\ref{Ex-F}).
By the definition of $\C{J}$ (see (\ref{vector-w-v})), it follows that
\begin{equation}\label{bill11}
(\m{w}_T-\m{w})\tr \mathcal{J}(\m{w}) =
\LAMBDA_T\tr\left(A\m{x}+B\m{y}-\m{b}\right)
-\LAMBDA\tr\left(A\m{x}_{T}+B\m{y}_{T}-\m{b}\right).
\end{equation}
For any $\m{w}^* = (\m{x}^*, \m{y}^*, \g{\lambda}^*) \in \C{W}^*$,
let us choose $\m{w} = (\m{x}^*, \m{y}^*, \g{\lambda})$,
where $\g{\lambda} = \g{\lambda}^* + \bar{\LAMBDA}$ and
$\bar{\LAMBDA}$ is a unit vector chosen so that
\[
(\bar{\LAMBDA})\tr \mathbb{E} \big[ A\m{x}_{T}+B\m{y}_{T}-\m{b}\big] =
- \mathbb{E} \big[ A\m{x}_{T}+B\m{y}_{T}-\m{b}\big] .
\]
Since $A\m{x}^*+B\m{y}^*=\m{b}$, this choice for $\m{w}$
in (\ref{bill11}) yields
\begin{equation}\label{bill87}
\mathbb{E}\big[(\m{w}_T-\m{w})\tr \mathcal{J}(\m{w})\big] =
\mathbb{E} \big[\|A\m{x}_{T}+B\m{y}_{T}-\m{b}\|
-(\LAMBDA^*)\tr(A\m{x}_{T}+B\m{y}_{T}-\m{b}) \big]
\end{equation}
Since $F(\m{w}) = F(\m{w}^*)$ when
$\m{w} = (\m{x}^*, \m{y}^*, \g{\lambda})$, (\ref{bill87}) yields
\begin{eqnarray}
&\mathbb{E}\big[F(\m{w}_T) - F(\m{w}) + (\m{w}_T-\m{w})\tr \mathcal{J}(\m{w})\big]
=& \label{bill86} \\
&\mathbb{E} \big[F(\m{w}_T) - F(\m{w}^*)
-(\LAMBDA^*)\tr(A\m{x}_{T}+B\m{y}_{T}-\m{b})
+ \|A\m{x}_{T}+B\m{y}_{T}-\m{b}\| \big] .
& \nonumber
\end{eqnarray}
By the variational inequality (\ref{Sec3-1}) with $\m{w} = \m{w}_T$, we have
\begin{eqnarray}
&F(\m{w}_T)- F(\m{w}^*) +(\m{w}_T -\m{w}^*)\tr \mathcal{J}(\m{w}^*) =&
\label{1st} \\
&F(\m{w}_T)- F(\m{w}^*)
-(\LAMBDA^*)\tr\left[A\m{x}_{T}+B\m{y}_{T}-\m{b}\right] \ge 0.& \nonumber
\end{eqnarray}
Use this inequality in (\ref{bill86}) to obtain the lower bound
\begin{equation}\label{bill88}
\mathbb{E}\big[F(\m{w}_T) - F(\m{w})
+ (\m{w}_T-\m{w})\tr \mathcal{J}(\m{w})\big] \ge
\mathbb{E}\big[\|A\m{x}_{T}+B\m{y}_{T}-\m{b}\|\big]
\end{equation}
for $\m{w} = (\m{x}^*, \m{y}^*, \g{\lambda})$.

Next, let us analyze the right side of (\ref{Ex-F}).
By the choice (\ref{Sec33-12ba2}) for $\eta_k$ and $M_k$, we see that
\[
\sum_{k = \kappa}^{\kappa + T} \eta_k M_k \le
\sum_{k = \kappa}^{\kappa + T}
c_1/(1 + c_3 k^\varrho ) .
\]
This sum is $\C{O}(1)$ if $\varrho > 1$, while it is $\C{O}(\log T)$
if $\varrho = 1$.
Since $\kappa$ was chosen so that $\eta_\kappa M_\kappa(M_\kappa+1) = c_1$,
it follows that the other terms in brackets on the right side of
(\ref{Ex-F}) are all $\C{O}(1)$ for
$\m{w} = (\m{x}^*, \m{y}^*, \g{\lambda})$
when $\g{\lambda} = \g{\lambda}^* + \bar{\LAMBDA}$ and
$\bar{\LAMBDA}$ is a unit vector.
Consequently, we have
\[
\mathbb{E}\big[F(\m{w}_T) - F(\m{w})
+ (\m{w}_T-\m{w})\tr \mathcal{J}(\m{w})\big] = E_\varrho (T).
\]
We combine this upper bound $E_\varrho (T)$ with
the lower bound (\ref{bill88}) to obtain
\begin{equation}\label{bill90}
\mathbb{E} \big[ \|A\m{x}_{T}+B\m{y}_{T}-\m{b}\| \big] = E_\varrho (T).
\end{equation}
This establishes the right side of (\ref{Ergodic-rate}).

The optimality condition (\ref{1st}) implies that
\begin{equation}\label{Elower}
\mathbb{E} \big[F(\m{w}_T)- F(\m{w}^*)\big] \ge
- \mathbb{E}\big[\|\LAMBDA^*\| \, \|A\m{x}_{T}+B\m{y}_{T}-\m{b}\|\big] =
-E_\varrho(T) .
\end{equation}
When $\m{w} = \m{w}^*$ in (\ref{Ex-F}), the right side of this inequality
is again $E_\varrho (T)$, so we have
\begin{equation}\label{Eupper}
\mathbb{E} \big[ F(\m{w}_T)- F(\m{w}^*)
+(\m{w}_T -\m{w}^*)\tr \mathcal{J}(\m{w}^*)\big] \le E_\varrho (T) .
\end{equation}
Together, (\ref{bill90}) and the equality in (\ref{1st}) yield
\begin{eqnarray}
&\mathbb{E} \big[ F(\m{w}_T)- F(\m{w}^*)
+(\m{w}_T -\m{w}^*)\tr \mathcal{J}(\m{w}^*)\big] \ge& \label{bill89} \\
&\mathbb{E} \big[ F(\m{w}_T)- F(\m{w}^*) \big]
- \mathbb{E} \big[ \|A\m{x}_{T}+B\m{y}_{T}-\m{b} \|\big] \|\LAMBDA^*\| =
\mathbb{E} \big[ F(\m{w}_T)- F(\m{w}^*) \big] - E_\varrho (T). \nonumber
\end{eqnarray}
The bound (\ref{Eupper}) and (\ref{bill89}) imply that
$\mathbb{E} \big[ F(\m{w}_T)- F(\m{w}^*)\big] \le  E_\varrho (T)$.
Combine this upper bound with the lower bound (\ref{Elower})
to obtain
$\left|\mathbb{E} \big[ F(\m{w}_T)- F(\m{w}^*)\big] \right| =  E_\varrho (T)$.
This establishes the left side of (\ref{Ergodic-rate}), which completes the
proof.
\end{proof}
\bigskip

We now have the following remarks.
\begin{remark}\label{remar1}
The objective error and the constraint violation converge to
zero in expectation due to $(\ref{Ergodic-rate})$;
however, this does not imply the convergence or boundedness of the
ergodic iterates.
If there exists $c > 0$ such that $\C{D}_k \succeq c \m{I}$,
then the iterates $(\m{x}_k, \m{y}_k, \g{\lambda}_k)$ are bounded
in expectation, and under a strong convexity assumption,
the ergodic iterates converge in expection (see Appendix).
\end{remark}

\begin{remark}
\label{remxx}
In AS-ADMM, it was required that
$\C{D}_k = \C{M}_k - \beta A \tr A  \succeq \m{0}$,
however, in Lemma~$\ref{Sec3-theore1}$, the proof only requires that
$\left\| \m{x}^{k+1} - \m{x}^k  \right\|_{\C{D}_k} \ge 0$
so that the third bracketed expression in $(\ref{bjc-39-new})$ can be
dropped while preserving the inequality.
Hence, for numerical efficiency, at any iteration $k$,
we could set $\C{M}_k = \rho_k \m{I}$ and  then adjust $\rho_k$
based on an underestimate $\beta \delta_2^k/\delta_1^k$
for the largest eigenvalue of $\beta A \tr A$, where
\[
\delta_1^k = \left\| \m{x}^k - \m{x}^{k-1}  \right\|^2 \quad \mbox{and} \quad
\delta_2^k = \left\| A (\m{x}^k - \m{x}^{k-1})  \right\|^2 .
\]
In particular,
given parameters $\rho_0$ and $\rho_{\min} > 0$, and $\eta > 1$,
we multiply $\rho_{\min}$ by $\eta$ in any iteration where
$ \rho_{k-1} < \beta \delta_2^k/\delta_1^k$, and
in each iteration, we set
\[
 \rho_k = \max \{ \rho_{\min}, \; \beta \delta_2^k/\delta_1^k \}.
\]
The increase in $\rho_{\min}$ can only happen a finite number of times
since $\rho_{k-1} \ge \beta \delta_2^k/\delta_1^k$
whenever $\rho_{k-1} \ge \beta  \|A \tr A\|$;
in fact, the increase in $\rho_{\min}$ can happen
at most $\lceil \log_{\eta} \frac{ \beta \|A \tr A\|}{ \rho_0} \rceil$ times.
Hence, for $k$ large enough, $\rho_k$, $\C{M}_k$, and $\C{D}_k$
are all unchanged, and
$\left\| \m{x}^{k+1} - \m{x}^k  \right\|_{\C{D}_k}^2$ $\ge 0$.
Related techniques were first used in $\cite{chy13}$ in the context of
a line search.
\end{remark}

\begin{remark}
Lemma $\ref{Sec3-theore1}$ holds under the  assumption that
$\left\{ \eta_kM_k(M_k+1)\right\}$ is nondecreasing.
We now point out that Lemma $\ref{Sec3-theore1}$
can be reformulated so as to hold when
$\left\{ \eta_kM_k(M_k+1)\right\}$ is {\it nonincreasing}
if $\C{X}$ is a bounded set.
Let $\C{N}_{\C{X}}$ denote the diameter of $\C{X}$:
\begin{equation}\label{CX}
\C{N}_{\C{X}} = \sup \{ \|\m{x}_1 - \m{x}_2 \|_{\C{H}}:
\m{x}_1, \m{x}_2 \in \C{X} \}.
\end{equation}
If $\C{N}_{\C{X}}$ is finite and
$\left\{ \eta_kM_k(M_k+1)\right\}$ is nonincreasing for
$k \in [\kappa, \kappa+T]$, then the term
$(\ref{Sec3-bcz-6})$ in the proof of Lemma $\ref{Sec3-theore1}$ has the
following bound: For any $\m{x} \in \C{X}$, we have
\begin{eqnarray}\label{Dec-M}
&& \sum\limits_{k=\kappa}^{\kappa+T}
\frac{2}{M_k(M_k+1)\eta_k}
\left( \|\m{x}-\breve{\m{x}}^{k}\|_{\C{H}}^2
-\|\m{x}- \breve{\m{x}}^{k+1}\|_{\C{H}}^2\right)\\
&=& \frac{2}{M_\kappa(M_\kappa+1)\eta_\kappa}
\|\m{x}- \m{x}^{\kappa}\|_{\C{H}}^2 - \frac{2}{M_T(M_T+1)\eta_T}\|\m{x}
- \breve{\m{x}}^{T+1}\|_{\C{H}}^2\nonumber\\
&&
- \sum\limits_{k=\kappa}^{\kappa + T-1}
\left(\frac{2}{M_k(M_k+1)\eta_k}
-\frac{2}{M_{k+1}(M_{k+1}+1)\eta_{k+1}}\right)
\|\m{x}- \breve{\m{x}}^{k+1}\|_{\C{H}}^2\nonumber \\
&\leq& \frac{2}{M_\kappa(M_\kappa+1)\eta_\kappa}
\mathcal{N}_{\mathcal{X}}^2
- \sum\limits_{k=\kappa}^{\kappa + T-1}\left(\frac{2}{M_k(M_k+1)\eta_k}
-\frac{2}{M_{k+1}(M_{k+1}+1)\eta_{k+1}}\right)
\mathcal{N}_{\mathcal{X}}^2\nonumber \\
&=& \frac{2}{M_T (M_T+1)\eta_T }\mathcal{N}_{\mathcal{X}}^2. \nonumber
\end{eqnarray}
By using $(\ref{Dec-M})$, a bound similar to $(\ref{Ex-F})$
in Lemma~$\ref{Sec3-theore1}$ can be established.
\end{remark}

\begin{remark}
If $N = 1$ and $\m{e}_t = \m{0}$, then AS-ADMM
is a deterministic ADMM with multiple accelerated gradient steps
to solve the $\m{x}$-subproblem inexactly, and the expectation operator
can be removed from $(\ref{Ergodic-rate})$.
If additional assumptions hold, such as $s \in (0, (1+\sqrt{5})/2)$
(the open interval), $\C{D}_k \succeq c\m{I}$ for some
$c > 0$, and $B$ has full column rank, then the iterates $\m{w}^k$ are
uniformly bounded and convergent to some $\m{w}^* \in \C{W}^*$.
\end{remark}

\section{Incremental Sampling of Stochastic Gradient with Variance
Reduction\label{extension:incremental}}
In this section, we discuss AS-ADMM algorithm with incremental sampling of
the stochastic gradient.
These techniques can potentially reduce the number of stochastic gradient
steps and can be beneficial when the subproblems for computing the
stochastic gradient step is expensive.
Suppose that at the $t$-th inner iteration of subroutine $\m{xsub}$
in the $k$-th outer iteration of AS-ADMM,
when calculating the stochastic gradient of function $f$,
we randomly select an index sample set
\[
U_t \subset\{1,2,\cdots,N\} \quad \textrm{of size} \quad |U_t|=m_k\leq N
\]
with uniform probability.
We define
\[
\hat{\m{g}}_t =   \frac{1}{m_k}\sum\limits_{i \in U_t } \nabla f_{i}(\hat{\m{x}}_t) \quad \mbox{and} \quad 
\m{e}_t = \nabla f( { \bar{\m{x}}^k })
- \frac{1}{m_k}\sum\limits_{i \in U_t } \nabla f_{i}(\bar{\m{x}}^k)
\]
for some choice of $\bar{\m{x}}^k \in \Omega$.
Since the elements of $U_t$ are chosen with uniform probability,
$\mathbb{E}\big[ \m{e}_t \big] = \m{0}$.
Also, we define
\begin{equation}\label{Sec33-bj1}
\m{d}_t= \hat{\m{g}}_t + \m{e}_t
= \frac{1}{m_k}\sum\limits_{i \in U_t }
\left[\nabla f_{i}(\hat{\m{x}}_{t})-\nabla f_{i}(\bar{\m{x}}^k)\right]
+ \nabla f(\bar{\m{x}}^k),
\end{equation}
and $\DELTA_t = \nabla f(\hat{\m{x}}_t) - \m{d}_t$.
Again, since $U_t$ is chosen with uniform probability,
we have $\mathbb{E}[ \DELTA_t ] = \m{0}$.
Moreover, if the diameter $\C{N}_\chi$ of $\C{X}$,
defined in (\ref{CX}), is finite, then
the variance of $\DELTA_t$ has the following bound:
\begin{eqnarray}
\mathbb{E} \left[ \|\DELTA_t\|_{\C{H}^{-1}}^2 \right] &=&
\mathbb{E}\left[ \left\|
\frac{1}{m_k}\sum\limits_{i \in U_t }\left[\nabla f_{i}(\bar{\m{x}}_{t})
-\nabla f_{i}(\hat{\m{x}}^k)\right]+
\nabla f(\hat{\m{x}}^k) - \nabla f(\bar{\m{x}}_t) \right\|_{\C{H}^{-1}}^2
\right] \nonumber\\
&= & \frac{N-m_k}{m_k(N-1)} \mathbb{E}_{\xi}\left[ \left\|
\nabla f_{\xi}(\bar{\m{x}}^k)-\nabla f_{\xi}(\hat{\m{x}}_{t})-
[\nabla f(\bar{\m{x}}^k)-\nabla f(\hat{\m{x}}_{t})]\right\|_{\C{H}^{-1}}^2 \right] \nonumber\\
&=& \frac{N-m_k}{m_k(N-1)} \left\{\mathbb{E}_{\xi}\left[ \left\|\nabla f_{\xi}(\bar{\m{x}}^k)-\nabla f_{\xi}(\hat{\m{x}}_{t})\right\|_{\C{H}^{-1}}^2\right]- \left\|\nabla f(\bar{\m{x}}^k)-\nabla f(\hat{\m{x}}_{t})\right\|_{\C{H}^{-1}}^2 \right\} \nonumber\\
&\le& \frac{1}{m_k}\mathbb{E}_{\xi}\left[ \left\|\nabla f_{\xi}(\bar{\m{x}}^k)-\nabla f_{\xi}(\hat{\m{x}}_{t})\right\|_{\C{H}^{-1}}^2\right] =  \frac{1}{m_k N} \sum\limits^{N}_{j=1}\left\|\nabla f_{j}(\bar{\m{x}}^k)-\nabla f_{j}(\hat{\m{x}}_{t})\right\|_{\C{H}^{-1}}^2 \nonumber\\
&\leq&
\frac{1}{m_k N} \sum\limits^{N}_{j=1}\nu^2
\left\| \bar{\m{x}}^k-\hat{\m{x}}_{t}\right\|_{\C{H}}^2
= \frac{\nu^2}{m_k}\left\| \bar{\m{x}}^k-\hat{\m{x}}_{t}\right\|_{\C{H}}^2 
\le \frac{\nu^2 \mathcal{N}_{\mathcal{X}}^2}{m_k}, \label{mk-variance}
\end{eqnarray}
where the second equality 
follows from \cite[Page 183]{Freund1962},
$\mathbb{E}_{\xi}[\cdot]$ is taken with respect to 
a random drawing of $\xi \in \{1,2,\ldots,N\}$ with uniform probability,
and $\nu$ is the Lipschitz constant for the $f_j$ given in
{\rm (a3)}.
Consequently, $\mathbb{E}[\|\DELTA_t\|_{\C{H}^{-1}}^2] \le \sigma^2$
with $\sigma =  \nu \mathcal{N}_{\mathcal{X}}/\sqrt{m_k}$.

On the other hand, if we obtain information during the computation
by choosing $\bar{\m{x}}^k$ such that
$\|\bar{\m{x}}^k - \hat{\m{x}}_t\|$ is small,
then we see from (\ref{mk-variance}) that
the variance of $\DELTA_t$ could be reduced significantly.
Note that the full gradient $\nabla f(\bar{\m{x}}^k) $ is 
only calculated in the outer iteration.
In our numerical experiments, we choose $\bar{\m{x}}^k$
to be the ergodic mean of the iterates at certain iterations.
Furthermore, under the conditions of Lemma \ref{Sec3-theore1},
we can show that
\begin{eqnarray}\label{Sec33-bj3}
&\mathbb{E}\left[ F(\m{w}_T)-F(\m{w})+(\m{w}_T-\m{w})\tr
\mathcal{J}(\m{w})\right]  \le \frac{1}{2 (1+T)} \bigg\{
(\nu\mathcal{N}_{\mathcal{X}})^2\sum\limits_{k=\kappa}^{\kappa+T}
\frac{\eta_k M_k}{m_k} &
\nonumber\\
&+ \|\m{w}-\m{w}^\kappa\|_{Q}^2
+ \xi_2 \left\| A \m{x}^\kappa + B \m{y}^\kappa
-\m{b} \right\|^2  + \frac{4}{M_\kappa(M_\kappa+1)\eta_\kappa} \|\m{x}
- \m{x}^{\kappa}\|_{\C{H}}^2 \bigg\},& \nonumber
  \end{eqnarray}
where $\xi_2 = \beta (1-s)^2$.

Suppose we choose the parameters
\begin{equation}\label{Sec33-bj4}
\eta_k= \eta \in \left(0, \frac{1}{2 \nu} \right], \quad M_k = M
\quad \mbox{and} \quad
m_k=\min\left\{  \lceil c (1+k)^{ \varrho } \rceil, N \right\},
\end{equation}
where $M \ge 1$ is an integer and $c >0$ and $\varrho \ge 1$ are real
scalars.
In the case that the total data size $N$ is large, with $m_k < N$,
we can deduce from (\ref{Sec33-bj3}) that
\begin{equation}\label{Sec33-bj33}
\mathbb{E}\left[ F(\m{w}_T)-F(\m{w})+(\m{w}_T-\m{w})\tr
\mathcal{J}(\m{w})\right] =
\C{O} \left(\frac{1}{T} \left(1+ \sum_{k=0}^{T}
\frac{1}{(1+k)^{\varrho}} \right) \right).
\end{equation}
Hence, the convergence rate with the incremental sampling of the
stochastic gradient will be the same
as the rate (\ref{Ergodic-rate}) of AS-ADMM
with parameter setting (\ref{Sec33-12ba2}).

In addition, it can be observed that with the parameter
settings (\ref{Sec33-bj4}), the total number of sample
gradients used in the inner iteration when $N$ is large and $m_k < N$
is given by
\[
\sum\limits_{k=0}^{T} M_k m_k = \mathcal{O} \left(\sum\limits_{k=0}^{T}(1+k)^{ \varrho } \right),
\]
which is on the same order as that of AS-ADMM
with parameter settings (\ref{Sec33-12ba2}).
However, the stepsize parameter $\eta_k$ in (\ref{Sec33-bj4}) can be  larger
than that in (\ref{Sec33-12ba2}),
and the total number of stochastic gradient steps performed in AS-ADMM is
\[
\sum\limits_{k=0}^{T}M_k  = M T,
\]
which  can be significantly smaller than the total number of
stochastic gradient steps
$ \mathcal{O} \left( \sum\limits_{k=0}^{T}k^{ \varrho } \right)$
performed by AS-ADMM with parameter settings (\ref{Sec33-12ba2});
this would greatly reduce the computational cost in the case that
the subproblem for calculating the stochastic gradient step is expensive.
\section{Linearized AS-ADMM\label{extension:linearized}} 
When $B$ is a relatively complicated matrix,
a closed-form solution of the $\m{y}$-subproblem may not exist,
even when $g$ is simple.
A common approach, in this case, is to modify
the $\m{y}$-subproblem by linearizing its quadratic
penalty term so that a closed-form solution may exist, similar to what
is done in {\tt xsub} for the $\m{x}$-subproblem.
The corresponding proximal term is
\[
\frac{1}{2}\left\|\m{y}-\m{y}^k\right\|_{\tau \m{I} - \beta B \tr B}^2,
\]
where $\tau > 0$ is large enough that $\tau \m{I} \succeq \beta B \tr B$. 
This proximal term, when added to the penalty term in the $\m{y}$-subproblem,
will annihilate the penalty term $(\beta/2)\|B\m{y}\|^2$.
For $\tau > 0$,
the $\m{y}$-subproblem reduces to the following proximal mapping:
\[
\m{y}^{k+1}=\m{prox}_{g, \tau}(\m{q}^k):=\arg\min\limits_{\m{y}\in\C{Y}}
\left\{ g(\m{y})+ (\tau/2) \| \m{y}- \m{q}^k\|^2 \right\},
\]
where $\m{q}^k = \m{y}^k - B\tr [
\beta (A\m{x}^{k+1} +B\m{y}^k -\m{b})- \LAMBDA^k ]/\tau$.
The complexity analysis when the $\m{y}$-subproblem is linearized is
the same as that of the original AS-ADMM given in Theorem~\ref{coll-2}
for appropriate choices of the parameters.
It may be possible to relax the constraint
$\tau \m{I} \succeq \beta B \tr B$ using ideas from \cite{CWHLv19, Tao2020}.

\section{Numerical Experiments}\label{sec-experiments}
This section provides numerical experiments to investigate the performance
of AS-ADMM.
\subsection{Test problem and parameter settings\label{tp}}
Given a number of training samples $\{(\m{a}_j,b_j)\}_{j=1}^{N}$
where $\m{a}_j\in \mathcal{R}^{l}$ and 
$b_j\in\{-1, 1\}$,  we solve the following  generalized lasso problem
(called the Graph-Guided Fused Lasso model):
\[
\min\limits_{\mathbf{x}}\frac{1}{N}\sum\limits_{j=1}^{N}f_j(\mathbf{x})
+\mu\|A\mathbf{x}\|_1,
\]
where  $f_j(\mathbf{x})=\log\left(1+\exp(-b_i \m{a}_i\tr \mathbf{x})\right)$
denotes the logistic loss function on the feature-label pair
$(\m{a}_j,b_j)$, $N$ is the data size (usually large),
$\mu>0$ is a given  regularization parameter,  
and $A = \m{I}$ or  $A=[\mathbf{G};\mathbf{I}]$,
where $\mathbf{G}$ is obtained from a sparse inverse covariance estimation
given in \cite{FHTR08}. 
Although the generalized lasso problem is used to compare the ADMM algorithms,
this specific problem is potentially solved more efficiently using a stochastic
primal-dual algorithm such as the one developed in \cite{Chambolle2018}.

By introducing an auxiliary variable $\mathbf{y}$,
the above problem can be reduced to a special case of problem (\ref{P}):
\begin{equation} \label{Sec5-prob}
\min\limits_{\mathbf{x,y}} \; \{F(\mathbf{x},\mathbf{y})=
\frac{1}{N}\sum\limits_{j=1}^{N}f_j(\mathbf{x})+\mu\|\mathbf{y}\|_1 :
A\mathbf{x}-\mathbf{y}=\mathbf{0} \}.
\end{equation}
We use AS-ADMM to solve (\ref{Sec5-prob});
the closed-form solutions of the subproblems are
\begin{equation}  \label{Sec5-bjc}
\left \{\begin{array}{lll}
\breve{\m{x}}_{t+1}&=&\left[\gamma_t \C{H}
+ \C{M}_k\right]^{-1}\left[\gamma_t \C{H}\breve{\m{x}}_{t}
+ \C{M}_k \mathbf{x}^k - \m{d}_t- \m{h}^k\right],\\
\mathbf{y}^{k+1}&=&\textrm{Shrink}\left(\frac{\mu}{\beta},  A\mathbf{x}^{k+1}-\frac{\g{\lambda}^k}{\beta}\right).
\end{array}\right.
\end{equation}
Here, $\textrm{Shrink}(\cdot,\cdot)$ denotes the
so-called soft shrinkage operator, which can be evaluated using the
built-in MATLAB function ``{\tt wthresh}''.

\renewcommand\figurename{Table}
\begin{figure}[h]
\begin{center}
\begin{center}
\caption {Real-world datasets and regularization parameters
used in the experiments \label{tab4.1}}
\end{center}
\begin{tabular}{|lcccccc|}
\hline
dataset && number of samples && dimensionality  && $\mu$ \\
\hline
{\tt a9a}    &&   32,561  &&  123  &&    1e-5\\
{\tt ijcnn1} &&   49,990  &&  23   &&      1e-5   \\
{\tt w8a}    &&   49,749 &&   300      &&     1e-5    \\
{\tt mnist}  &&   11,791 &&  784   &&   1e-5    \\
\hline
\end{tabular}
\end{center}
\end{figure}
\renewcommand\figurename{Fig.}

The datasets of Table~\ref{tab4.1}
and the Lipschitz constants of $f$ are taken from the LIBSVM website. 
The parameter settings used in AS-ADMM are as follows.
The stepsize $s$ is taken as $s = 1.618$ (approximately its largest value),
the penalty parameter is $\beta = 0.04$, and
the values of $\eta_k$ and $M_k$ are given by
(\ref{Sec33-12ba2}) with $c_1 = 1/\nu$, $c_2 = 1/(2\nu)$,
$c_3 = 0.01$, $\varrho = 1.1$, and $M = 200$.
The choice for $c_2$ ensures that the condition $\eta_k \in (0, 1/2\nu]$
of Lemma~\ref{Sec3-theore1} is satisfied from the start of the iterations,
while $c_1$ was chosen so that it scaled in the same way as $c_2$.
A small value was used for $c_3$ so that the growth of $M_k$ would be delayed,
and $M = 200$ since $N$ is on the order of tens of thousands, and we wanted
at least several hundred inner iterations.
The matrices $ \C{M}_k$ are updated adaptively by the strategy
in Remark~\ref{remxx} with initial values $\rho_0 = 1$, $\eta = 1.1$,
$\rho_{\min} = 10^{-5}$;
these were the same parameter values that seemed to work well in
\cite{HagerZhang19,HagerZhang19b} when we solved image reconstruction
problems.
In particular, $\eta = 1.1$ so that the lower bound $\rho_{\min}$
for the largest eigenvalue would grow slowly.
We set $\mathcal{H}:=\sigma\mathbf{I}$ with $\sigma = 2\times 10^{-5}$.
Thus both $\C{M}_k$ and $\C{H}$ are diagonal matrices.
We set the regularization parameter $\mu = 10^{-5}$ since
ASVRG-ADMM set all the regularization parameters to $10^{-5}$.
We found that it is expensive and unnecessary to calculate
one full gradient at each outer iteration for reducing the variance
of the stochastic gradient.
Hence, in numerical experiments, we only do the variance reduction
when the number of inner iterations $M_k$ is larger than the dimension of
the $\m{x}$-variable.
More precisely, at the $k$-th outer iteration of AS-ADMM,
in the $t$-th inner iteration of subroutine $\m{xsub}$, we set 
\[
\m{e}_t =
\left \{\begin{array}{ll}
\nabla f(\m{x}_{k-1}) - \nabla f_{\xi_t} (\m{x}_{k-1})
& \mbox{if } ~ M_k >  n_1, \\
\m{0}, & \mbox{otherwise},
\end{array}\right.
\]
where $\m{x}_k$ is the ergodic mean of the $\m{x}$-iterates.
All comparison algorithms are implemented in MATLAB R2018a (64-bit)
with the same starting point
$(\mathbf{x}^0,\mathbf{y}^0,\LAMBDA^0)=(\mathbf{0},\mathbf{0},\mathbf{0})$,
and all experiments are performed on a PC with Windows 10 operating system,
with an Intel i7-8700K CPU, and with 16GB RAM.

\subsection{Comparative Experiments}
In this section, we compare the following algorithms for solving
problem (\ref{Sec5-prob}) using the four data sets of Table~\ref{tab4.1}:
\begin{itemize}
\item
Accelerated stochastic ADMM, Algorithm \ref{algo1} (\textbf{AS-ADMM}).
\item
Stochastic ADMM (\cite{Ouyang13}, \textbf{STOC-ADMM}).
\item
Accelerated variance reduced stochastic ADMM (\cite[Alg. 2]{LiuShangCheng17},
\textbf{ASVRG-ADMM}).
\item
Accelerated Linearized ADMM with $\chi=1$(\cite[Alg. 2]{OuyangLan15},
\textbf{ALP-ADMM}).
\item
The classic ADMM \cite{GM75} with $f$ linearized
(\textbf{L-ADMM}):
\[
\m{x}^{k+1}=\arg\min\limits_{\m{x} \in \C{R}^l}  \left\langle\nabla f(\m{x}^k), \m{x}- \m{x}^k\right\rangle +\frac{\nu}{2}\left\|\m{x}- \m{x}^k\right\|^2+  \frac{\beta}{2}
\left\| A\m{x}-\m{y}^k -\LAMBDA^k/\beta \right\|^2.
\]
\end{itemize}
We did not compare AS-ADMM with many other stochastic algorithms mentioned
in this paper since their performance has been shown in the
literature to be worse than that of ASVRG-ADMM. 
We compare AS-ADMM with STOC-ADMM \cite{Ouyang13}
since STOC-ADMM only applies one 
stochastic gradient step to solve the $\m{x}$-subproblem in
each outer iteration, while AS-ADMM applies
a multiple number of accelerated gradient steps as determined by the theory,
and ASVRG-ADMM preforms a fixed number $m = N/200$ inner iterations.
Note that both ALP-ADMM and L-ADMM are deterministic ADMM-type algorithms
using the full gradient.

In comparing algorithms, we plot Opt\_err,
the maximum of the relative objective
error (Obj\_err) and constraint violation (Equ\_err),
versus CPU time in seconds, where
\[
\mbox{Obj\_err} =
\frac{|F({\m{x}}, {\m{y}})-F^*|}{\max\{F^*,1\}} \qquad \mbox{and} \qquad
\mbox{Equ\_err} = \left\| A {\m{x}} -{\m{y}} \right\|
\]
Here $F^*$ is the approximate optimal objective function value 
obtained by running AS-ADMM for more than 10 minutes.  
For AS-ADMM, we plot the error associated with the iterates over the
first 1/3 of the total CPU time budget, followed by the error associated with
the ergodic iterates over the last $2/3$ of the budget.
Note that the convergence theory describes the error for $k \ge \kappa$, where
$\kappa$ is the iteration number where the assumptions in the
analysis are satisfied.
An advantage of AS-ADMM is that the algorithm is completely adaptive,
and the user does not need to provide Lipschitz constants or eigenvalue bounds;
and in theory, convergence is guaranteed.
Nonetheless, the initial iterates may be less reliable than later iterates.

Figures \ref{FigProb1-1}-\ref{FigProb1-3} show results
for the data sets {\tt a9a}, {\tt ijcnn1} and {\tt w8a} and $A = \m{I}$,
while Figure~\ref{FigProb2}
is the corresponding plot for the {\tt mnist} data set with the more
complicated choice $A=[\mathbf{G}; \mathbf{I}]$ explained
in subsection~\ref{tp}.
We can see that both  AS-ADMM and  ASVRG-ADMM perform better than  
STOC-ADMM \cite{Ouyang13},
where only one stochastic gradient step is used in each iteration
to solve the $\m{x}$-subproblem.
We also see that AS-ADMM and  ASVRG-ADMM achieve
comparable performance on the lasso problems for the first
three data sets, while AS-ADMM performs significantly better than 
ASVRG-ADMM on the last data set, where the constraint is more complex.
Note that $\mbox{Opt\_err}$ for AS-ADMM
has a big drop at around $1/3$ of the CPU time budget, the point where
we start to utilize the ergodic iterates when reporting the objective value.
Observe that both stochastic algorithms, AS-ADMM and ASVRG-ADMM,
perform significantly better than the deterministic methods
ALP-ADMM and L-ADMM, while the accelerated nature of
ALP-ADMM leads to much better performance than that of the classic L-ADMM.

\begin{figure}[htbp]
 \begin{minipage}{1\textwidth}
 \def\figurename{\footnotesize Fig.}
 \centering
\resizebox{9cm}{6cm}{\includegraphics{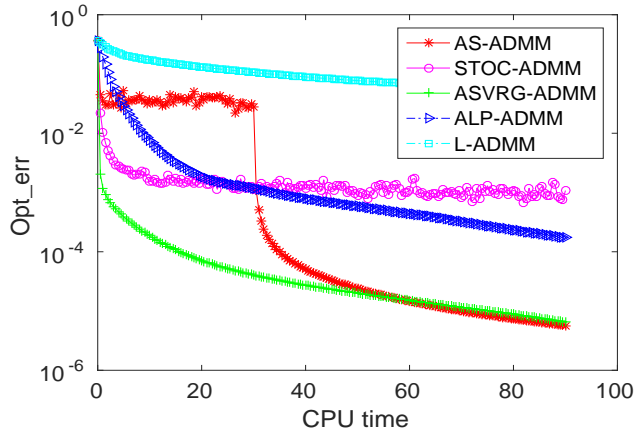}}
\caption{\footnotesize  Comparison of $\mbox{Opt\_err}$ vs CPU time   
for Problem $(\ref{Sec5-prob})$ and the {\tt a9a} dataset } \label{FigProb1-1}
   \end{minipage}
\end{figure}

\begin{figure}[htbp]
 \begin{minipage}{1\textwidth}
 \def\figurename{\footnotesize Fig.}
 \centering
\resizebox{8.1cm}{6cm}{\includegraphics{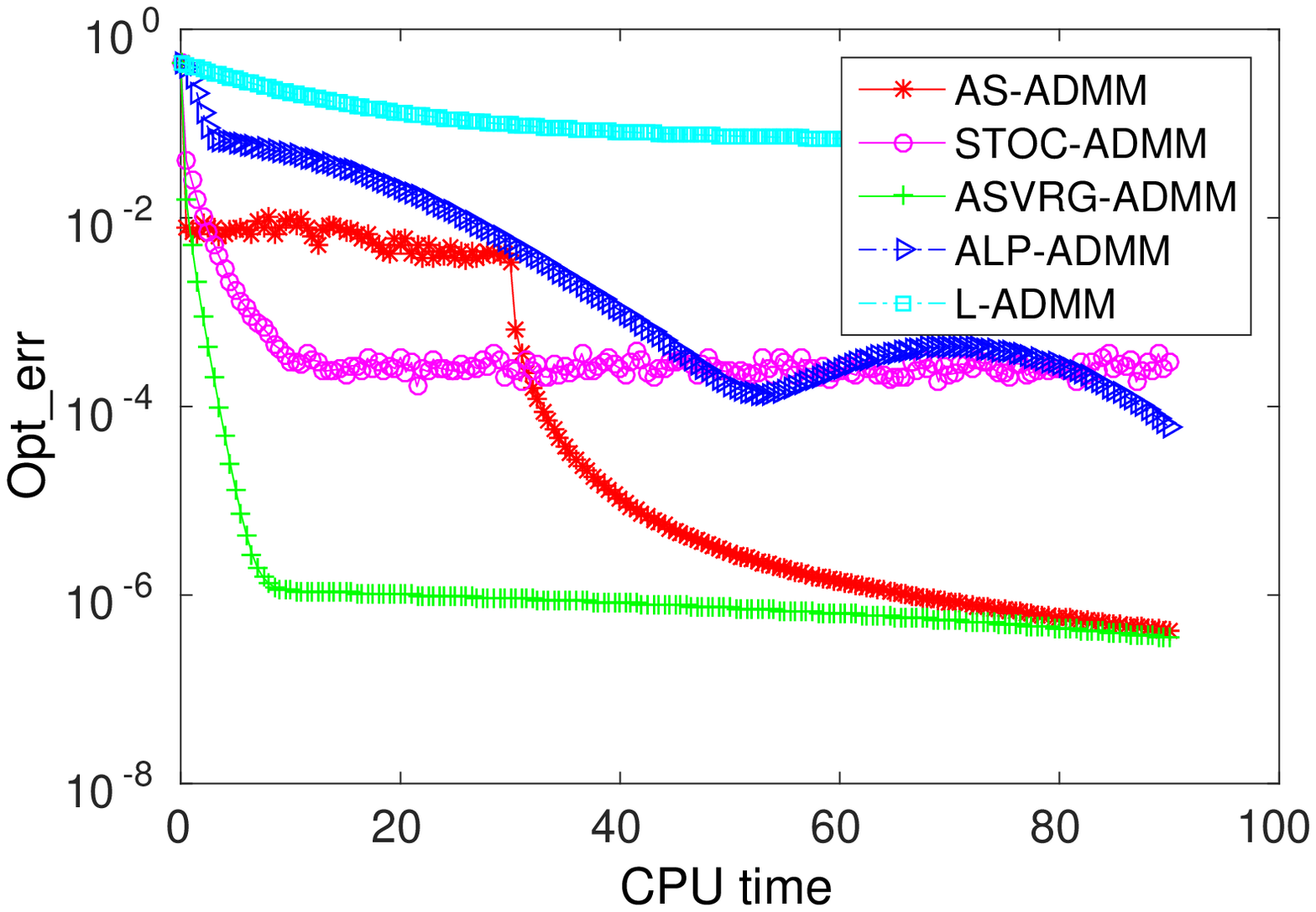}}
\caption{\footnotesize  Comparison of $\mbox{Opt\_err}$ vs CPU time   for  Problem $(\ref{Sec5-prob})$ and the {\tt ijcnn1} dataset } \label{FigProb1-2}
   \end{minipage}
\end{figure}

\begin{figure}[htbp]
 \begin{minipage}{1\textwidth}
 \def\figurename{\footnotesize Fig.}
 \centering
\resizebox{9cm}{6cm}{\includegraphics{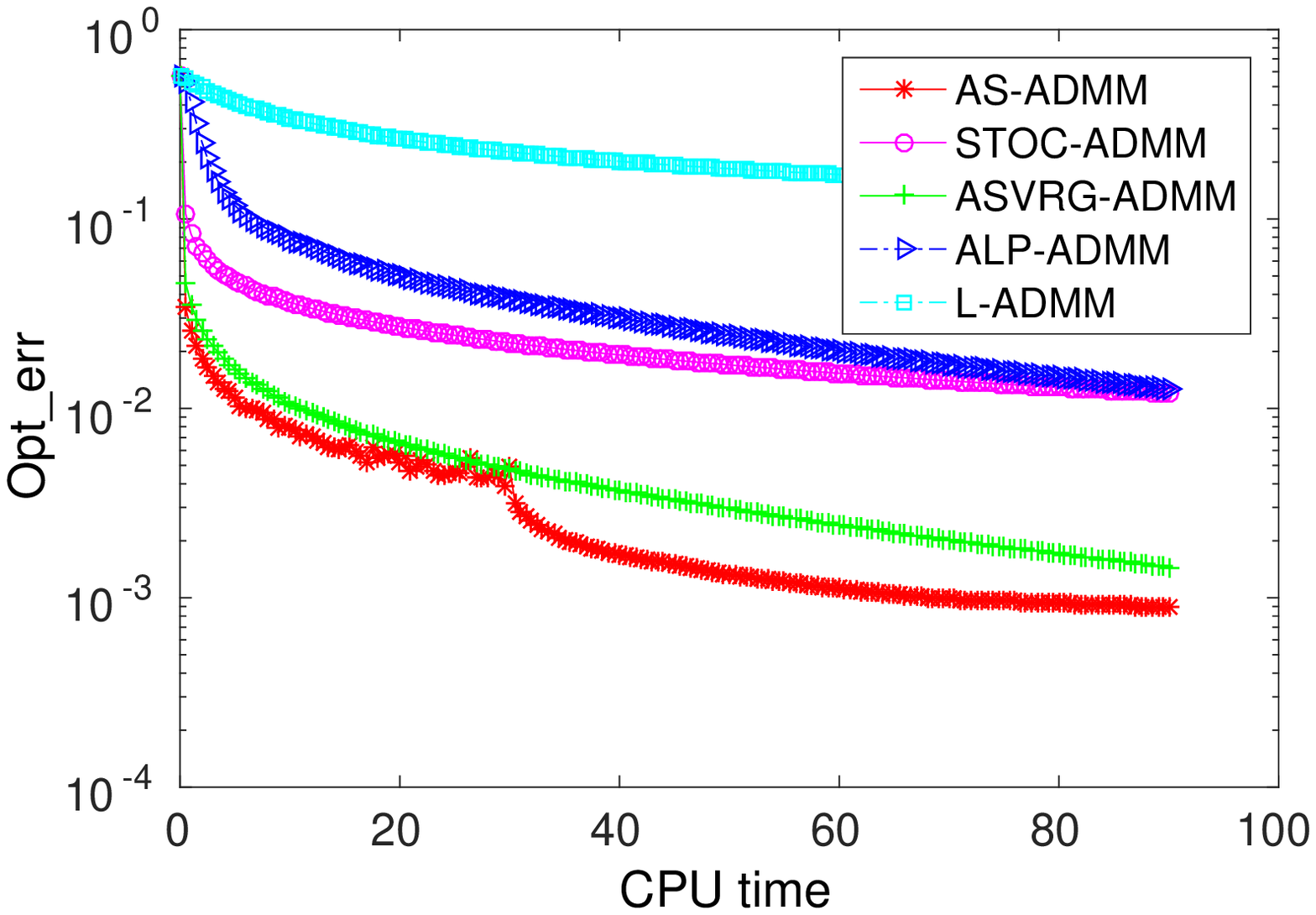}}
\caption{\footnotesize  Comparison of $\mbox{Opt\_err}$ vs CPU time   for   Problem $(\ref{Sec5-prob})$ and the  {\tt w8a} dataset } \label{FigProb1-3}
   \end{minipage}
\end{figure}

\begin{figure}[htbp]
 \begin{minipage}{1\textwidth}
 \def\figurename{\footnotesize Fig.}
 \centering
\resizebox{9cm}{6cm}{\includegraphics{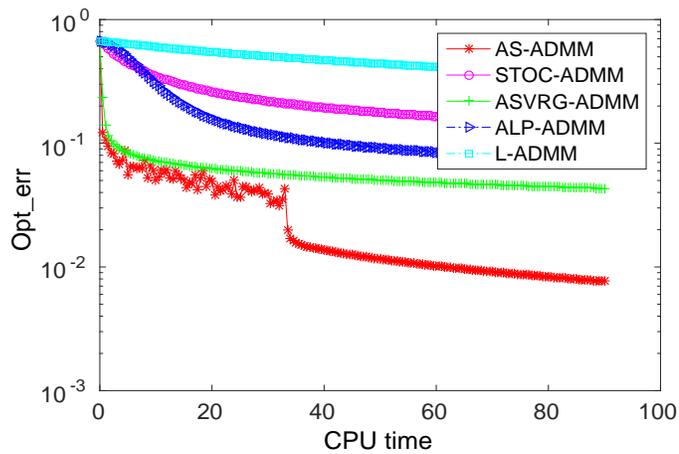}}
\caption{\footnotesize  Comparison of $\mbox{Opt\_err}$ vs CPU time for Problem $(\ref{Sec5-prob})$ and the {\tt mnist} dataset } \label{FigProb2}
   \end{minipage}
\end{figure}

\section{Conclusion}\label{sec-conc}
We have developed an accelerated stochastic ADMM for
solving a type of regularized empirical risk minimization
problem arising in machine and statistical learning. 
We also discussed incremental sampling techniques,
which are potentially beneficial when the subproblems for computing the
stochastic gradient step are expensive, and a
variant of AS-ADMM that was achieved by linearizing the $\m{y}$-subproblem.
The proposed algorithm AS-ADMM combines
both the variance reduction technique and an
accelerated gradient method for fast convergence.
Using a unified variational analysis, the expected objective error and
constraint violation for ergodic iterates are $\C{O}(1/k)$ or
$\C{O}(k^{-1}\log k)$, depending on the choice of parameters.
Numerical experiments on group lasso problems using
well-established stochastic and deterministic ADMM algorithms show that
AS-ADMM can be very effective for solving data mining and machine
learning problems with large data sets.
With stronger assumptions, bounds for the iterates in expectation, as well
as linear convergence results, are established.
\section{Appendix: Additional Properties of the Iterates}\label{appendix}
In the appendix, we derive additional properties of AS-ADMM which
involve new assumptions that do not appear in the previous analysis.
\subsection{Iteration Bounds \label{bounds}}
When $\C{D}_k := \C{M}_k - \beta A'A$ is uniformly positive definite,
the expectation of the (nonergodic) iterates
$\m{w}^k = (\m{x}^k, \m{y}^k, \g{\lambda}^k)$
is uniformly bounded.
\smallskip

\begin{proposition} \label{P1}
If {\rm (A1)} and {\rm (A2)} are satisfied, the parameters
$\eta_k$ and $M_k$ are chosen according to $(\ref{Sec33-12ba2})$ with
$\varrho > 1$\, and there exists $c > 0$ such that $\C{D}_k \succeq c\m{I}$
for every $k$, then $\mathbb{E}\big[\|\m{w}^k\|^2\big]$ is bounded
uniformly in $k$;
moreover, if $s \in (0, (1+\sqrt{5})/2)$, then
$\mathbb{E}\big[ \|A\m{x}^k + B\m{y}^k - \m{b}\|^2 \big]$ tends to $0$,
while $\mathbb{E}\big[ \|A\m{x}^k + B\m{y}^k - \m{b}\|^2 \big]$
is uniformly bounded if $s = (1+\sqrt{5})/2$.
\end{proposition}
\smallskip

\begin{proof}
Insert $\m{w} = \m{w}^* \in \C{W}^*$ in (\ref{bjc-39-new}) and utilize
(\ref{Sec3-1}) to obtain
\begin{equation}\label{initialbound}
\left\|\m{w}^{k+1} - \m{w}^*\right\|^2_{Q_{k+1}}
-\left\|\m{w}^{k} - \m{w}^*\right\|^2_{Q_{k}} \le
\xi_2 (\gamma_k - \gamma_{k+1}) - \xi_1 \gamma_{k+1} - 2\zeta^k ,
\end{equation}
where $\gamma_k = \|A\m{x}^k + B\m{y}^k - \m{b}\|^2$,
$\zeta^k$ is defined in (\ref{zeta_k}) with $\m{x} = \m{x}^*$,
$\xi_2 := \beta (1-s)^2 \ge 0$, and
$\xi_1 := \beta ((2-s) - (1-s)^2)  > 0$ if $s \in (0, (1+\sqrt{5})/2)$.
Let $E_k$ be defined by
\[
E_k = 
\left\|\m{w}^k-\m{w}^*\right\|^2_{Q_{k}} +
\xi_2 \gamma_k + \frac{4}{M_k (M_k+1)\eta_k}
\left\|\breve{\m{x}}^{k}-\m{x}^*\right\|_{\C{H}}^2 .
\]
Since $M_k (M_k+1) \eta_k$ is nondecreasing, it follows from
(\ref{initialbound}) and the definition of $\zeta^k$ that
\[
E_{k+1} - E_{k} + \xi_1\gamma_{k+1} \le
\frac{4}{M_k(M_k+1)}\bigg[
\sum\limits_{t=1}^{M_k}t\langle \DELTA_t, \breve{\m{x}}_{t}
-\m{x}^*\rangle +\frac{\eta_k}{4(1-\eta_k\nu)}
\sum\limits_{t=1}^{M_k}t^2\left\|\DELTA_t\right\|_{\C{H}^{-1}}^2\bigg].
\]
As shown in the proof of Lemma~\ref{Sec3-theore1},
the expection of the first term on the right side vanishes,
while the expectation of the second term is bounded by $\sigma^2 \eta_k M_k$.
Hence, we have
\[
\mathbb{E} \big[ E_{k+1} - E_{k} + \xi_1\gamma_{k+1} \big] \le
\sigma^2 \eta_k M_k .
\]
We sum this inequality over $k \in [\kappa, j)$ to obtain
\[
\mathbb{E} \big[ E_j \big] + \xi_1 \sum_{k = \kappa+1}^j
\mathbb{E} \big[ \gamma_k \big] \le \mathbb{E} \big[ E_\kappa \big]
+ \sigma^2 \sum_{k = \kappa}^{j-1} \eta_k M_k .
\]
Since $\varrho > 1$, it follows that $\eta_k M_k$ is summable
when (\ref{Sec33-12ba2}) holds,
and $\mathbb{E} \big[ E_{j} \big]$ is bounded, uniformly in $j$.
Moreover, when $s \in (0, (1+\sqrt{5})/2)$,
the bound for the sum of $\mathbb{E} \big[\gamma_k\big]$ over $k \ge \kappa$
implies that $\mathbb{E} \big[\gamma_k \big]$ tends to zero.
If $s = (1+\sqrt{5})/2$, then $\xi_2 > 0$ and the uniform bound for
$\mathbb{E} \big[ E_{j} \big]$ implies that
$\mathbb{E} \big[\gamma_j \big]$ is uniformly bounded.
\end{proof}

\subsection{Convergence of Ergodic Iterates Under Strong Convexity}
We now show that an error bound such as (\ref{Ergodic-rate})
implies convergence of the ergodic iterates in expectation
when strong convexity holds.
\smallskip

\begin{proposition} \label{P2}
Suppose $(\ref{Ergodic-rate})$ holds.
If either $f$ and $g$ are strongly convex or $f$ is strongly convex and
the columns of $B$ are linearly independent, then
\[
\mathbb{E} \big[\|\m{x}_T - \m{x}^*\|^2 + \|\m{y}_T - \m{y}^*\|^2 \big] =
E_\varrho(T),
\]
where $(\m{x}^*, \m{y}^*)$ is the unique solution of $(\ref{P})$ and
$E_\varrho(T)$ is defined in Theorem~$\ref{coll-2}$.
\end{proposition}
\smallskip

\begin{proof}
If $f$ is strongly convex with modulus $\alpha$, then it follows from
strong convexity, the first-order optimality conditions for a stationary point
$(\m{x}^*, \m{y}^*, \g{\lambda}^*)$, and the inclusion
$\m{x}_T \in \C{X}$ that
\begin{eqnarray*}
f(\m{x}_T) - f(\m{x}^*) &\ge& \nabla f(\m{x}^*)(\m{x}_T - \m{x}^*)
+ \frac{\alpha}{2} \|\m{x}_T - \m{x}^*\|^2 \\
&\ge& (\g{\lambda}^*)\tr A(\m{x}_T - \m{x}^*)
+ \frac{\alpha}{2} \|\m{x}_T - \m{x}^*\|^2.
\end{eqnarray*}
If $g$ were also strongly convex, with the same modulus $\alpha$,
then the same inequality holds, but with
$\m{x}$ replaced by $\m{y}$, $A$ replaced by $B$, and gradient replaced
by subgradient.
Together, these inequalities yield
\[
F(\m{w}_T) - F(\m{w}^*) \ge
(\g{\lambda}^*)\tr \left[ A\m{x}_T + B\m{y}_T - \m{b} \right]
+ \frac{\alpha}{2} \left[\|\m{x}_T - \m{x}^*\|^2 + 
\|\m{y}_T - \m{y}^*\|^2 \right] .
\]
Taking expectations and utilizing (\ref{Ergodic-rate}) gives
\begin{eqnarray}
\mathbb{E} \big[\|\m{x}_T - \m{x}^*\|^2 + 
\|\m{y}_T - \m{y}^*\|^2 \big] &\le&
\frac{2}{\alpha}
\mathbb{E} \big[ \|A\m{x}_T + B\m{y}_T - \m{b} \| \big]
\|\g{\lambda}^*\| + E_\varrho(T)
\label{x-bound} \\
&=& E_\varrho (T) . \nonumber
\end{eqnarray}

On the other hand, if $g$ is only convex, not strongly convex, then
the $\|\m{y}_T - \m{y}^*\|^2$ term in this last inequality is lost.
But if the columns of $B$ were linearly independent, then the equation
error can be manipulated as follows:
\begin{eqnarray*}
\mathbb{E} \big[\left\|A\m{x}_{T}+B\m{y}_{T}-\m{b} \right\|\big] &=&
\mathbb{E} \big[\left\|A(\m{x}_{T}-\m{x}^*) +
B(\m{y}_{T}-\m{y}^*)\right\| \big] \\
&\ge&
\mathbb{E} \big[\left\| B(\m{y}_{T}-\m{y}^*)\right\| \big] -
\mathbb{E} \big[\left\| A(\m{x}_{T}-\m{x}^*)\right\| \big] .
\end{eqnarray*}
Thus we have
\[
\mathbb{E} \big[\left\| B(\m{y}_{T}-\m{y}^*)\right\| \big] \le
E_\varrho (T) + 
\mathbb{E} \big[\left\| A(\m{x}_{T}-\m{x}^*)\right\| \big].
\]
Hence, the bound for $\mathbb{E} \big[\|\m{x}_T - \m{x}^*\|^2 \big]$
from (\ref{x-bound}) and the independence of the columns in $B$ imply
again that $\mathbb{E} \big[\|\m{y}_T - \m{y}^*\|^2 \big] = E_\varrho (T)$.
\end{proof}
\smallskip

\subsection{Linear Convergence of Iterates Under Strong Convexity}
In this section, it is proved that AS-ADMM is linearly convergent
when $f$ and $g$ are strongly convex.
The analysis requires a geometric growth rate for the inner iterations,
similar to the geometric growth rate employed in \cite{XieShanbhag20}
for the analysis of a much different ADMM.
In detail, the linear convergence result makes the following assumptions:
\smallskip

\begin{itemize}
\item[(L1)]
$f$ and $g$ are strongly convex with modulus $\alpha > 0$, and
both $f$ and $g$ have Lipschitz continuous gradients, with Lipschitz
constant $\nu$.
\item[(L2)]
The sets $\C{X} = \mathbb{R}^{n_1}$ and $\C{Y} = \mathbb{R}^{n_2}$.
\item[(L3)]
The stepsize $s \in (0, (1+\sqrt{5})/2)$ and
for some $c_1$ and $c_2 > 0$ and for all $k$, we have
\[
c_1 \m{I} \preceq \C{D}_{k+1} \le \C{D}_k :=
\C{M}_k - \beta A\tr A \preceq c_2 \m{I} .
\]
Moreover, $\mathbb{E}\big[\|\DELTA_t\|_{\C{H}^{-1}}^2\big] \le \sigma^2$
for some $\sigma > 0$,
independent of $t$ and the iteration number $k$, where $\DELTA_t$
is defined in $(\ref{1-sec1-05})$. 
\item[(L4)]
For some $\theta > 0$, we have
\[
M_k = \lceil (1+\theta)^2 M_{k-1}(\|\breve{\m{x}}^k\|^2 + 1)\rceil
\quad \mbox{and} \quad \eta_k = (1+\theta)^{-k}/M_k.
\]
\end{itemize}
\smallskip

By (L4), $M_k (M_k+1) \eta_k$ is nondecreasing, and we have:
\smallskip

\begin{itemize}
\item[(L4a)]
$M_k \ge (1+\theta)^{2k}$,
\item[(L4b)]
$M_k \eta_k = (1+\theta)^{-k}$, and
\item[(L4c)]
$\|\breve{\m{x}}^k\|^2/[M_k(M_k+1) \eta_k] \le (1+\theta)^{-k}$.
\end{itemize}
\smallskip

\begin{proposition}
If {\rm (L1)--(L4)} hold, then there exists $c > 0$ and
$0 < \tau < 1$ such that
\[
\mathbb{E} \big[ \|\m{x}^k - \m{x}^*\|^2 + \|\m{y}^k - \m{y}^*\|^2
+ \|\g{\lambda}^k - \bar{\g{\lambda}}^k\|^2 \big]
\le c\tau^k
\]
for all $k \ge 0$, where $\bar{\g{\lambda}}$ denotes
the projection of $\g{\lambda}$ onto
$\g{\Lambda}^*$, the set of all multipliers associated with the solution
$(\m{x}^*, \m{y}^*)$ of $(\ref{P})$.
\end{proposition}
\smallskip

\begin{proof}
Throughout the proof, $c$ denotes a generic positive constant,
independent of $k$, which typically has different values in different equations.
If $\bar{\m{w}}^k :=$
$(\m{x}^*, \m{y}^*, \bar{\g{\lambda}}^k)$, then by
the strong convexity assumption and the first-order optimality conditions
for $(\m{x}^*, \m{y}^*)$, we have
\begin{eqnarray*}
F(\tilde{\m{w}}^k) - F(\bar{\m{w}}^k) +
( \tilde{\m{w}}^k - \bar{\m{w}}^k)\tr \mathcal{J}(\bar{\m{w}}^k) 
&=& F(\tilde{\m{w}}^k) - F(\bar{\m{w}}^k) +
(A\m{x}^{k+1} + B\m{y}^{k+1} - \m{b})\tr \bar{\g{\lambda}}^k \\
&\ge& \frac{\alpha}{2} \left( \|\m{x}^{k+1} - \m{x}^*\|^2 +
\|\m{y}^{k+1} - \m{y}^*\|^2 \right) .
\end{eqnarray*}
Hence, by (\ref{bjc-39-new}) with $\m{w} = \bar{\m{w}}^k$ and
$\gamma_k = \|A\m{x}^k + B\m{y}^k - \m{b}\|^2$, we have
\begin{eqnarray}
&\alpha \left( \|\m{x}^{k+1} - \m{x}^*\|^2 + \|\m{y}^{k+1} - \m{y}^*\|^2
\right) \le \|\m{w}^{k} - \bar{\m{w}}^k\|_{Q_{k}}
 - \|\m{w}^{k+1} - \bar{\m{w}}^k\|_{Q_{k+1}}& \nonumber \\
& + \xi_2 (\gamma_k - \gamma_{k+1}) - \|\m{x}^{k+1} - \m{x}^k\|^2
- \xi_1 \gamma_{k+1} - 2 \zeta^k,& \label{*2}
\end{eqnarray}
where $\zeta^k$ is defined in (\ref{zeta_k}) with $\m{x} = \m{x}^*$,
and $\xi_1, \xi_2 \ge 0$.
Since
$\|\g{\lambda}^{k+1} - \bar{\g{\lambda}}^{k+1}\| \le$
$\|\g{\lambda}^{k+1} - \bar{\g{\lambda}}^{k}\|$, it follows that
\begin{equation}\label{*3}
\|\m{w}^{k+1} - \bar{\m{w}}^{k+1}\|_{Q_{k+1}} \le
\|\m{w}^{k+1} - \bar{\m{w}}^{k}\|_{Q_{k+1}}.
\end{equation}
Add $0.5\xi_1(\gamma_k - \gamma_{k+1})$ to each side of (\ref{*2}) and
combine with (\ref{*3}) to obtain
\[
\alpha \left( \|\m{x}^{k+1} - \m{x}^*\|^2 + \|\m{y}^{k+1} - \m{y}^*\|^2
\right) + \|\m{x}^{k+1} - \m{x}^k\|^2 + .5\xi_1(\gamma_k + \gamma_{k+1})
\le E_k - E_{k+1} - 2 \zeta^k ,
\]
where $E_k := \|\m{w}^k - \bar{\m{w}}^k\|_{Q_k}^2 +
(.5\xi_1 + \xi_2) \gamma_k$ and $\xi_1 > 0$ since $s \in (0, (1+\sqrt{5})/2)$.
The left side of this inequality is bounded from below by a positive
multiple $\bar{c}$ of
\[
d_k =
\|\m{x}^{k+1} - \m{x}^*\|^2 + \|\m{y}^{k+1} - \m{y}^*\|^2
+ \|\m{x}^{k+1} - \m{x}^k\|^2 + \gamma_k + \gamma_{k+1} .
\]
Hence, the inequality can be rearranged to yield
$E_{k+1} \le E_{k} - 2\zeta^k - \bar{c} d_k $.
We will show that $|\mathbb{E}\big[ \zeta^k \big]| \le c(1+\theta)^{-k}$,
in which case, we have
\begin{equation}\label{finalE}
\mathbb{E} \big[ E_{k+1}\big] \le
\mathbb{E} \big[ E_{k}\big] + c (1+\theta)^{-k}
- \bar{c} \mathbb{E} \big[ d_k \big] .
\end{equation}
Note that $\mathbb{E} \big[ d_k \big]$ must approach zero.
Otherwise, there exists $\epsilon > 0$ and an infinite number of indices
$k$ where $\mathbb{E} \big[ d_k \big] \ge \epsilon$.
If this were to hold, then 
$\mathbb{E} \big[ E_{k}\big]$ is eventually negative, which is impossible.

For $\m{x} = \m{x}^*$, (\ref{zeta_k}) gives $\zeta^k =$
\[
\frac{\Gamma_k}{\eta_k}\left[ \left\|\m{x}^*
-\breve{\m{x}}^{k}\right\|_{\C{H}}^2-\left\|\m{x}^*
-\breve{\m{x}}^{k+1}\right\|_{\C{H}}^2\right]
+\Gamma_k \sum\limits_{t=1}^{M_k}t\langle \DELTA_t,
\breve{\m{x}}_{t}-\m{x}^*\rangle
+\frac{\Gamma_k \eta_k }{4(1-\eta_k\nu)}
\sum\limits_{t=1}^{T}t^2\left\|\DELTA_t\right\|_{\C{H}^{-1}}^2,
\]
where $\Gamma_k = 2/[M_k(M_k+1)]$.
By (L4c), for the given choice of $M_k$ and $\eta_k$, the first bracketed
term in $\zeta^k$ is bounded by $c(1+\theta)^{-k}$.
Also, by (L4b), we have $M_k \eta_k =$ $(1+\theta)^{-k}$.
Taking the expectation of $\zeta^k$ and utilizing the estimates
obtained in Lemma~\ref{Sec3-theore1} for the last two terms
in $\zeta^k$, we obtain a bound
of the form $\left| \mathbb{E} \big[ \zeta^k\big] \right| \le c(1+\theta)^{-k}$.

To complete the proof, we will show that
\begin{equation}\label{complete}
\mathbb{E}\big[ E_{k+1} \big] \le
c (\mathbb{E} \big[ d_k \big] +(1+\theta)^{-k}).
\end{equation}
This is combined with the bound (\ref{finalE}) to obtain for some $r > 0$,
\[
\mathbb{E}\big[ E_{k+1} \big] \le \frac{1}{1+r}
\mathbb{E}\big[ E_{k} \big] + c(1+\theta)^{-k}.
\]
Consequently, $\mathbb{E}\big[ E_{k} \big]$ converges to zero at linear
convergence rate $\tau$ where
\[
1 > \tau > \max\{1/(1+r), 1/(1+\theta) \}.
\]

To establish (\ref{complete}), first define $e_1$ and $e_2$ by
\[
e_1 (\m{x}, \g{\lambda}) =
\|\nabla f(\m{x}) - A\tr \g{\lambda}\|
\quad \mbox{and} \quad
e_2 (\m{y}, \g{\lambda}) =
\|\nabla g(\m{y}) - B\tr \g{\lambda}\| .
\]
Also, define $\bar{\m{x}}^k = \arg \min \{\phi_k (\m{x}): \m{x} \in \C{X}\}$,
where $\phi_k$ is defined in (\ref{phi-psi-def}).
Since $\bar{\m{x}}^k$ minimizes $\phi_k$ over $\C{X} = \mathbb{R}^{n_1}$,
$\nabla \phi_k(\bar{\m{x}}^k) = \m{0}$, which implies that
\[
\nabla f(\bar{\m{x}}^k) -
A \tr \g{\lambda}^k + \beta A\tr
(A\m{x}^k + B\m{y}^k - \m{b}) + \C{M}_k(\bar{\m{x}}^k - \m{x}^k) = \m{0}.
\]
Utilize this equality in the definition of $e_1$ and exploit
the Lipschitz continuity of $\nabla f$ to obtain
\begin{equation}\label{e1xk+1}
e_1 (\m{x}^{k+1}, \g{\lambda}^{k}) \le
c \left( \sqrt{\gamma_k} + \|\bar{\m{x}}^k - \m{x}^k\|
+ \|\bar{\m{x}}^k - \m{x}^{k+1}\| \right) ,
\end{equation}
Inserting $\m{x} = \bar{\m{x}}^k$ and $T = M_k$ in (\ref{1-sec1-13}) gives
\begin{eqnarray}
& \left\|{\m{x}}^{k+1} -\bar{\m{x}}^k\right\|_{\C{M}_k}^2
\le \frac{2\Gamma_k}{\eta_k}\left[
\left\|\bar{\m{x}}^k -\breve{\m{x}}^{k}\right\|_{\C{H}}^2-\left\|\bar{\m{x}}^k
-\breve{\m{x}}^{k+1}\right\|_{\C{H}}^2\right] & \nonumber \\
& +2\Gamma_k \sum\limits_{t=1}^{M_k}t\langle \DELTA_t,
\breve{\m{x}}_{t}-\bar{\m{x}}^k\rangle
+\frac{2\Gamma_k \eta_k }{4(1-\eta_k\nu)}
\sum\limits_{t=1}^{T}t^2\left\|\DELTA_t\right\|_{\C{H}^{-1}}^2,&
\label{bill81}
\end{eqnarray}
where $\Gamma_k = 2/[M_k(M_k+1)]$.
The right side of (\ref{bill81}) was the same expression $\zeta^k$
that was analyzed previously, except that $\m{x}^*$ is now replaced by
$\bar{\m{x}}^k$.
By the previous analysis, the last two terms on the right side of
(\ref{bill81}) are bounded by $c(1+\theta)^{-k}$ in expectation.
In the first term, observe that
\[
\left\|\bar{\m{x}}^k -\breve{\m{x}}^{k}\right\|_{\C{H}}^2 \le
2 \left( \left\|\bar{\m{x}}^k\right\|_{\C{H}}^2 +
\left\|\breve{\m{x}}^k\right\|_{\C{H}}^2 \right) .
\]
Again, by the previous analysis,
$(\Gamma_k/\eta_k)\|\breve{\m{x}}^k\|_{\C{H}}^2 \le c(1+\theta)^{-k}$
due to (L4c).
We will show that
$\mathbb{E}\big[\|\bar{\m{x}}^k\|_{\C{H}}^2\big]$ is uniformly bounded,
which implies that
$(\Gamma_k/\eta_k)\mathbb{E}\big[\|\bar{\m{x}}^k\|_{\C{H}}^2\big] \le$
$c(1+\theta)^{-k}$.
In this case, (L3) and (\ref{bill81}) yield
\[
c_1 \mathbb{E}\left[ \left\|\m{x}^{k+1} - \bar{\m{x}}^k\right\|^2 \right] \le
\mathbb{E}\left[ \left\|\m{x}^{k+1} - \bar{\m{x}}^k\right\|_{\C{M}_k}^2 \right]
\le c (1+\theta)^{-k}.
\]
This is combined with (\ref{e1xk+1}) to obtain
\begin{eqnarray}
\mathbb{E} \big[ e_1 (\m{x}^{k+1}, \g{\lambda}^{k})^2 \big] &\le&
c \left( \mathbb{E} \big[ \gamma_k 
+ \|\bar{\m{x}}^k - \m{x}^k\|^2 \big]
+ (1+\theta)^{-k}  \right)
\label{e1xk} \\
&\le&
c \left( \mathbb{E} \big[ \gamma_k 
+ \|\m{x}^{k+1} - \m{x}^k\|^2 \big]
+ (1+\theta)^{-k}  \right) . \nonumber
\end{eqnarray}

To obtain a bound for $\mathbb{E}\big[\|\bar{\m{x}}^k\|_{\C{H}}^2\big]$,
we utilize strong convexity (L1) along with any $\m{x} \in \mathbb{R}^{n_1}$
to obtain
\[
0 \ge \phi_k(\bar{\m{x}}^k) - \phi_k(\m{x}) \ge
\nabla \phi_k(\m{x})(\bar{\m{x}}^k - \m{x}) + \frac{\alpha}{2}
\|\bar{\m{x}}^k - \m{x}\|^2 .
\]
Hence, by the Schwarz and triangle inequalities, we have
\[
\|\bar{\m{x}}^k\| - \|\m{x}\| \le
\|\bar{\m{x}}^k - \m{x}\| \le \frac{2}{\alpha} \|\nabla \phi_k(\m{x})\| =
\frac{2}{\alpha} \|\nabla f(\m{x}) + \m{h}^k + \C{M}_k (\m{x} - \m{x}^k)\| .
\]
For $\m{x} = \m{x}_f^*$, the minimizer of $f$, we have
$\nabla f(\m{x}_f^*) = \m{0}$ and
\[
\|\bar{\m{x}}^k\|^2 \le
c \left( \|\m{h}^k\|^2 + \|\m{x}^k\|^2 + \|\m{x}_f^*\|^2 \right) .
\]
Since $\|\m{h}^k\|^2$ and $\|\m{x}^k\|^2$
are uniformly bounded in expectation
by Proposition~\ref{P1}, it follows that
$\|\bar{\m{x}}^k\|^2$ is uniformly bounded in expectation.

Similar to the treatment of $\bar{\m{x}}^k$,
the optimality condition for $\m{y}^{k+1}$,
the solution of the $\m{y}$-subproblem in AS-ADMM, is
\[
\nabla g(\m{y}^{k+1}) - B \tr \g{\lambda}^k
+ \beta B\tr (A\m{x}^{k+1} + B\m{y}^{k+1} -\m{b}) = \m{0}.
\]
This identity is rearranged to give
\begin{equation} \label{e2yk}
e_2(\m{y}^{k+1}, \g{\lambda}^k) \le c \sqrt{\gamma_{k+1}}, \quad
\mbox{which yields} \quad
\mathbb{E}\big[ e_2 (\m{y}^{k+1}, \g{\lambda}^k)^2 \big]
\le c \mathbb{E}\big[ \gamma_{k+1} \big] .
\end{equation}

The set of multipliers $\g{\Lambda}^*$ are those $\g{\lambda} \in \mathbb{R}^n$
that satisfy both of the equations
$A\tr \g{\lambda} = \nabla f(\m{x}^*)$ and
$B\tr \g{\lambda} = \nabla g(\m{y}^*)$.
Hence, $\g{\Lambda}^*$ is a particular solution plus any vector in the
null space $\C{N}$ of the matrix $[ A \; \; B ]\tr$.
The projection $\bar{\g{\lambda}}$ on $\g{\Lambda}^*$ has the property that
$\g{\lambda} - \bar{\g{\lambda}}$ is orthogonal to $\C{N}$, which implies
the existence of a constant $c$ such that
\[
\|\g{\lambda} - \bar{\g{\lambda}}\|^2
\le c \|[A\;\; B]\tr(\g{\lambda} - \bar{\g{\lambda}})\|^2  =
c \left( \|A\tr (\g{\lambda} - \bar{\g{\lambda}})\|^2 +
\|B\tr (\g{\lambda} - \bar{\g{\lambda}})\|^2 \right) .
\]
Since $A\tr \bar{\g{\lambda}} = \nabla f(\m{x}^*)$ and
$B\tr \bar{\g{\lambda}} = \nabla g(\m{y}^*)$,
this bound can be rewritten
\begin{equation}\label{lambda}
\|\g{\lambda} - \bar{\g{\lambda}}\|^2 \le
c \left( e_1 (\m{x}^*, \g{\lambda})^2 + 
e_2 (\m{y}^*, \g{\lambda})^2 \right) .
\end{equation}

The following inequality is deduced from the triangle inequality,
and the Lipschitz assumption for the gradient of $f$:
\begin{eqnarray*}
e_1(\m{x}^*,\g{\lambda})^2 &\le&
2e_1(\m{x}^{k+1},\g{\lambda}^k)^2
+ 2\left( e_1(\m{x}^*, \g{\lambda})
- e_1(\m{x}^{k+1}, \g{\lambda}^k) \right)^2\\
&\le&
c \left( e_1(\m{x}^{k+1},\g{\lambda}^k)^2
+ \|\m{x}^{k+1} - \m{x}^*\|^2
+ \|\g{\lambda}^k - \g{\lambda}\|^2 \right) .
\end{eqnarray*}
An analogous inequality holds for $e_2$.
Hence, by (\ref{lambda}),
\begin{equation}\label{lambdabound}
\|\g{\lambda} - \bar{\g{\lambda}}\|^2 \le c \delta_{k+1},
\end{equation}
where
\[
\delta_{k+1} =
e_1(\m{x}^{k+1},\g{\lambda}^k)^2 +
e_2(\m{y}^{k+1},\g{\lambda}^k)^2 +
\|\m{x}^{k+1} - \m{x}^*\|^2 +
\|\m{y}^{k+1} - \m{x}^*\|^2 +
\|\g{\lambda}^{k} - \g{\lambda}\|^2 .
\]

Now insert $\g{\lambda} = \g{\lambda}^{k+1}$ in (\ref{lambdabound}),
take expectation, and utilize the bounds (\ref{e1xk}) and (\ref{e2yk})
to obtain
\[
\mathbb{E} \big[ \|\g{\lambda}^{k+1} - \bar{\g{\lambda}}^{k+1}\|^2 \big] \le
c \left( \mathbb{E}\big[ d_{k}\big] + (1 + \theta)^{-k} \right).
\]
Since
\[
\|\m{w}^{k+1} - \bar{\m{w}}^{k+1}\|_{Q_{k+1}}^2 =
\|\m{x}^{k+1} -\m{x}^*\|^2_{\C{D}_{k+1}} + \beta\|B(\m{y}^{k+1} - \m{y}^*)\|^2
+ \|\g{\lambda}^{k+1} - \bar{\g{\lambda}}^{k+1}\|^2/(s\beta) \\
\]
and
\[
E_{k+1} = \|\m{w}^{k+1} - \bar{\m{w}}^{k+1}\|_{Q_{k+1}}^2 
+ 0.5(\xi_1 + \xi_2) \gamma_{k+1},
\]
it follows that
$\mathbb{E} \big[ E_{k+1} \big] \le$
$c ( \mathbb{E}\big[ d_{k}\big] + (1 + \theta)^{-k})$,
which completes the proof of (\ref{complete}).
\end{proof}
\bibliographystyle{siam}

\end{document}